%% file: article.tex
\documentclass[10pt,a4paper, english]{smfart}

\usepackage[utf8]{inputenc}
\usepackage[english, french]{babel}
\usepackage[T1]{fontenc}
\usepackage{lmodern}

\usepackage[hidelinks,pdfencoding=unicode]{hyperref}
\usepackage[shortcuts]{extdash}
\usepackage{xpatch}

\usepackage{amsmath}
\usepackage{amssymb}
\usepackage{stmaryrd}
\usepackage{mathtools} % \MoveEqLeft

\usepackage[all,2cell]{xy}
\UseAllTwocells
\SilentMatrices
\SelectTips{cm}{10}

\usepackage{enumitem}
\setlist{nosep}
\setlist[enumerate, 1]{label={\rm (}\emph{\alph*}{\rm )}}
\setlist[enumerate, 2]{label={\rm (}\emph{\alph{enumi}.\arabic*}{\rm )}}

\makeatletter

\patchcmd\@begintheorem
  {\ignorespaces}
  {\hypertarget{\csname @currentHref\endcsname}{}\ignorespaces}
  {}{}

\patchcmd\@maketitlehook
  {\ifx\@empty\@keywords\else\@footnotetext{\@setkeywords}\fi}
  {\ifx\@empty\@keywords\else\@footnotetext{\@setkeywords}\fi
   \ifx\@empty\@altkeywords\else\@footnotetext{\@setaltkeywords}\fi}
  {}{}

\makeatother

\input{macros}

\frfalse

\author{Dimitri Ara}
\address{Aix~Marseille~Univ,~CNRS,~Centrale~Marseille,~I2M,~Marseille,~France}
\email{dimitri.ara@univ-amu.fr}
\urladdr{\href{http://www.i2m.univ-amu.fr/perso/dimitri.ara/}{http://www.i2m.univ-amu.fr/perso/dimitri.ara/}}

\title{A Quillen Theorem B for strict $\infty$-categories}

\begin{document}

\frontmatter

\begin{abstract}
We prove a generalization of Quillen's Theorem B to strict \oo-categories.
More generally, we show that under similar hypothesis as for
Theorem~B, the comma construction for strict \oo-categories, that we
introduced with Maltsiniotis in a previous paper, is the homotopy pullback
with respect to Thomason equivalences. We give several applications of these
results, including the construction of new models for certain Eilenberg--Mac
Lane spaces.
\end{abstract}



\maketitle

\mainmatter

\section*{Introduction}

This paper is a part of an ongoing research project with Maltsiniotis and
Gagna about the homotopy theory of $\ooCat$, the category of strict
\oo-categories and strict \oo-functors, including for the moment the papers
and preprints \cite{AraMaltsiNThom, AraMaltsiCondE, AraMaltsiJoint,
AraMaltsiThmAI, AraMaltsiThmAII, Gagna}. By ``homotopy theory of $\ooCat$'',
we mean the study of strict \oo-categories through their classifying spaces,
defined by means of the so-called Street nerve \cite{StreetOrient},
associating a simplicial set to every strict \oo-category. In other words,
the homotopy theory of~$\ooCat$ is the homotopy theory of the pair $(\ooCat,
\W_\infty)$, where $\W_\infty$ denotes the class of \ndef{Thomason
equivalences} of $\ooCat$, that is, of strict \oo-functors sent to simplicial weak
homotopy equivalences by Street's nerve. Gagna proved in \cite{Gagna} that
the localization of $\ooCat$ by $\W_\infty$ is equivalent to the homotopy
category of spaces.  Therefore, the study of the homotopy theory of $\ooCat$
is the study of the homotopy theory of spaces from an alternative point of
view. We refer to the introduction of \cite{AraMaltsiThmAI} for a detailed
exposition of this project.

But let us step back.  This homotopy theory of strict \oo-categories is
inspired by the homotopy theory of $\Cat$, the category of small categories,
as developed by Quillen~\cite{QuillenHAKTI}, Thomason~\cite{Thomason} and
Grothendieck~\cite{GrothPS} (see also~\cite{Maltsi, Cisinski}). The starting
point of this theory is the idea of Quillen to define his higher algebraic
K-theory groups in terms of homotopy types of categories. To study the
resulting theory, he introduced his famous Theorems A and B,
establishing the main properties of \ndef{Thomason equivalences} of $\Cat$,
that is, functors sent to simplicial weak homotopy equivalences by the usual
nerve functor.

Before recalling the statement of Theorem B, let us introduce some notation
and terminology. If $u : A \to B$ is a functor and $b$ is an object of $B$,
we will denote by $\cotr{A}{b}$ the category of objects of $A$ under $b$,
that is, of pairs $(a, f : b \to u(a))$, where $a$ is an object of $A$ and
$f$ is a morphism of $B$. In other words,
the category $\cotr{A}{b}$ is the comma category $b \comma u$, where $b$ is
identified with the functor from the terminal category to $B$ of value~$b$.
We will say that a functor $u : A \to B$ is \ndef{colocally homotopically
constant} if, for every morphism $f : b \to b'$ of $B$, the functor
  $\cotr{A}{f} :
\cotr{A}{b'} \to \cotr{A}{b}$ induced by $f$ is a Thomason equivalence.

\begin{theoremB}[Quillen]
  If $u : A \to B$ is a colocally homotopically constant functor, then, for
  every object $b$ of $B$, the category $\cotr{A}{b}$ is canonically the homotopy
  fiber of $u$ at $b$.
\end{theoremB}

The homotopy fiber of the statement has to be understood in the sense of the
Thomason model category structure \cite{Thomason} but can also be
interpreted by simply saying that $N(\cotr{A}{b})$ is the simplicial
homotopy fiber of $Nu$ at $b$, where $N$ denotes the nerve functor. This
theorem was generalized by Barwick and Kan \cite{BarwickKanThmB} to the
following statement about homotopy pullbacks:

\begin{theorem*}[Barwick--Kan]
  Consider a diagram
  \[
    \xymatrix{
      A \ar[r]^u & C & B \ar[l]_v
    }
  \]
  of $\Cat$, where $v$ is colocally homotopically constant.
  Then the comma construction~$u \comma v$ is canonically the homotopy
  pullback $A \timesh_C B$.
\end{theorem*}

This result is also a consequence of a version of Quillen's Theorem B due to
\hbox{Cisinski} \cite[Theorem 6.4.15]{Cisinski}\footnote{More precisely,
this follows from condition $(vi)$ of the ``dual'' of Cisinski's theorem
applied to the canonical functor $u \comma C \to C$, which is a
  Grothendieck cofibration, so that any functor is transverse to it
in the sense of Cisinski (and from the fact that the canonical functor $A
\to u \comma C$ is a deformation retract and hence a Thomason
equivalence).}.
Barwick and Kan actually proved a similar result for (weak) $(\infty,
n)$-categories. Note that this generalization, for $n > 0$, is orthogonal to
our project as the weak equivalences used by Barwick and Kan are the
equivalences of $(\infty, n)$-categories and not the Thomason equivalences.
When $n = 0$, these two classes of weak equivalences coincide and the
above theorem is essentially the case $n = 0$.

The theorems of Quillen and Barwick--Kan for categories were generalized to
$2$\nbd-categories endowed with Thomason equivalences, and even
bicategories, by \hbox{Cegarra} \forlang{et al.}~\cite{CegThmB, CegThmBBicat,
CegComma}.

\medbreak

The purpose of this paper is to generalize these two theorems to strict
\oo-categories. We introduced, with Maltsiniotis, in our study of the
\oo-categorical Theorem A \cite{AraMaltsiThmAII}, a comma construction for
strict \oo-categories. As for categories, a special case of this
construction allows to define a slice construction and one can thus define
the notion of a colocally homotopically constant strict \oo-functor.
Therefore, using Street's nerve to define homotopy pullbacks, the
statements of these two theorems still make sense for strict \oo-categories
and our main result can be stated as:

\begin{theorem*}
  Consider a diagram
  \[
    \xymatrix{
      A \ar[r]^u & C & B \ar[l]_v
    }
  \]
  of $\ooCat$, where $v$ is colocally homotopically constant.
  Then the comma construction~$u \comma v$ is canonically the homotopy
  pullback $A \timesh_C B$.
\end{theorem*}

\goodbreak

This implies the Theorem B for strict \oo-categories:

\begin{theorem*}
  If $u : A \to B$ is a colocally homotopically constant strict \oo-functor,
  then, for every object $b$ of $B$, the \oo-category $\cotr{A}{b}$ is
  canonically the homotopy fiber of $u$ at $b$.
\end{theorem*}

{
\tolerance=500
The proof is inspired by Cisinski's proof of the original Quillen Theorem~B
\hbox{\cite[Theorem 6.4.15]{Cisinski}} and is in particular based on a
simplicial result due to Rezk that can be thought of as a simplicial version
of Quillen's Theorem B. Besides Rezk's result, our proof rely mainly on two
tools. First, we use the sesquifunctoriality of the comma construction that
we proved with \hbox{Maltsiniotis}
\hbox{\cite[Appendix~B]{AraMaltsiThmAII}}. Second, we use Steiner's theory
of augmented directed complexes~\cite{Steiner} to produce ``contractions''
of Street's orientals~\cite{StreetOrient}.
}

As an application of our Theorem B, we prove the following statement
about loop spaces of strict \oo-categories:

\begin{theorem*}
  Let $A$ be a strict \oo-category endowed with an object $a$. Suppose that for
  every $1$\nbd-cell $f : a' \to a''$ the induced \oo-functor $\Homi_A(a'',
  a) \to \Homi_A(a', a)$ is a Thomason equivalence. Then
  $\Homi_A(a,a)$ is a model for the loop space of $(A, a)$.
\end{theorem*}

Using this theorem and a particular case of the Theorem A for strict
\oo-categories~\cite{AraMaltsiThmAI,AraMaltsiThmAII}, that we deduce from
Theorem B, we obtain new models for some Eilenberg--Mac Lane spaces:

\begin{theorem*}
  Let $\pi$ be a commutative ordered group whose underlying poset is
  directed. Denote by $\pi^+$ its monoid of positive elements. Then, for any
  $n \ge 1$, the \oo-category $\Kcat{\pi^+}{n}$ is a $\Ktop{\pi}{n}$.
\end{theorem*}

In this statement, $\Kcat{\pi^+}{n}$ denotes the obvious strict $n$-category
having only one $i$-cell for $0 \le i < n$ and whose set of $n$-cells is
$\pi^+$. In particular, we get that $\Kcat{\N}{n}$ is~a~$\Ktop{\Z}{n}$.

\medskip

Our paper is organized as follows. The first section contains simplicial
preliminaries and in particular Rezk's simplicial Theorem B. In the
second section, we introduce the \oo-categorical notions that we need. We
recall the main properties of the Gray tensor product and of oplax
transformations.  We give a brief overview of the theory of comma
\oo-categories that we introduced with Maltsiniotis in
\cite{AraMaltsiThmAII}. We use this theory to define slice \oo-categories
and a kind of mapping space factorization for strict \oo-functors. The third
section contains the main results of the paper. We introduce the notion of
Thomason equivalences,  homotopy pullback squares in $\ooCat$ and colocally
homotopically constant strict \oo-functors. We study the homotopical
behavior of these \oo-functors under base change, from which we deduce our
main theorems, including the Theorem B for strict \oo-categories. The fourth
section contains several applications of our Theorem B. We prove the
non-relative case of the Theorem A for strict \oo-categories. We then apply
Theorem B to study loop spaces of strict \oo-categories. Using our results,
we produce new models for certain Eilenberg--Mac Lane spaces. We end the
section with an application to loop spaces of strict \oo-groupoids. Finally,
in an appendix, we produce, using Steiner's theory \cite{Steiner}, the
``contractions'' of Street's orientals \cite{StreetOrient} needed in the
proof of the main result.

\section{Simplicial preliminaries}

\begin{paragraph}
  We will denote by $\cDelta$ the \ndef{simplex category}. Recall that it is the
  full subcategory of the category of ordered sets whose objects are the
  \[ \Deltan{n} = \{0 < \cdots < n \} \]
  for $n \ge 0$. The category of \ndef{simplicial sets}, that is, of presheaves
  over $\cDelta$, will be denoted by $\pref{\cDelta}$. The Yoneda embedding
  $\cDelta \hookto \pref{\cDelta}$ will always be considered as an
  inclusion.
\end{paragraph}

\begin{paragraph}\label{paragr:simpl_we}
  By a \ndef{weak equivalence} of simplicial sets, we will always mean a
  weak homotopy equivalence, that is, a weak equivalence of the
  Kan--Quillen model category structure. Similarly, by a homotopy pullback
  square of simplicial sets, we will always mean a homotopy pullback square
  for the Kan--Quillen model category structure.
\end{paragraph}

\goodbreak

The following proposition can be considered as a kind of simplicial Theorem
B and will be crucial to our proof of the \oo-categorical Theorem B:

\begin{proposition}[Rezk]\label{prop:simpl_thmB}
  Let $p : X \to Y$ be a morphism of simplicial sets. The following
  conditions are equivalent:
  \begin{enumerate}
    \item\label{cond:pullback} every pullback square
      \[
        \xymatrix{
          X' \pullbackcorner \ar[r] \ar[d] & X \ar[d]^p \\
          Y' \ar[r] & Y
        }
      \]
      is a homotopy pullback square,
    \item\label{cond:fib} for every diagram of pullback squares
      \[
        \xymatrix{
          X'' \pullbackcorner \ar[d] \ar[r]^{u'} & X' \pullbackcorner \ar[r] \ar[d] &
            X \ar[d]^p \\
          Y'' \ar[r]_u & Y' \ar[r] & Y \pbox{,}
        }
      \]
      if $u$ is a weak equivalence, then so is $u'$,
    \item\label{cond:loc_const} for every diagram of pullback squares
      \[
        \xymatrix{
          X'' \pullbackcorner \ar[d] \ar[r]^{u'} & X' \pullbackcorner \ar[r] \ar[d] &
            X \ar[d]^p \\
          \Deltan{n} \ar[r] & \Deltan{m} \ar[r] & Y \pbox{,}
        }
      \]
      the morphism $u'$ is a weak equivalence,
    \item\label{cond:loc_const_zero} for every diagram of pullback squares
      \[
        \xymatrix{
          X'' \pullbackcorner \ar[d] \ar[r]^{u'} & X' \pullbackcorner \ar[r] \ar[d] &
            X \ar[d]^p \\
          \Deltan{0} \ar[r] & \Deltan{m} \ar[r] & Y \pbox{,}
        }
      \]
      the morphism $u'$ is a weak equivalence.
  \end{enumerate}
\end{proposition}

\begin{proof}
  The equivalence between conditions~\ref{cond:pullback} and \ref{cond:fib}
  follows from \cite[Proposition 2.7]{RezkFib}. The equivalence between
  these two conditions and condition \ref{cond:loc_const} is a consequence
  of \cite[Theorem 4.1]{RezkFib} (see \cite[Remark 4.2]{RezkFib}). Clearly,
  condition~\ref{cond:loc_const} implies condition~\ref{cond:loc_const_zero}
  and it suffices to prove the converse. Consider a diagram of pullback
  squares as in condition~\ref{cond:loc_const} and form the diagram of
  pullback squares
  \[
    \xymatrix{
      X''' \pullbackcorner \ar[r]^{u''} \ar[d] &
      X'' \pullbackcorner \ar[d] \ar[r]^{u'} &
       X' \pullbackcorner \ar[r] \ar[d] & X \ar[d]^p \\
      \Deltan{0} \ar[r]_0 &
      \Deltan{n} \ar[r] & \Deltan{m} \ar[r] & Y \pbox{.}
    }
  \]
  By condition~\ref{cond:loc_const_zero}, the morphisms $u''$ and $u'u''$
  are weak equivalences. This implies that $u'$ is a weak equivalence,
  thereby proving the result.
\end{proof}

\goodbreak

We end the section with a probably well-known fact about fiber products of
strong deformation retracts.

\begin{paragraph}
  By a \ndef{strong left deformation retract}, we will mean a
  triple
  \[ (i : A \to B, r : B \to A, h : \Deltan{1} \times B \to B) \]
  of simplicial maps such that
  \begin{enumerate}
    \item $r$ is a retraction of $i$ (so that $ri = \id{A})$,
    \item $h$ is a homotopy from $ir$ to $\id{B}$,
    \item we have $h(\Deltan{1} \times i) = ip_2$, where $p_2 : \Deltan{1}
      \times A \to A$ denotes the second projection.
  \end{enumerate}
  If the homotopy $h$ goes from $\id{B}$ to $ir$, instead of going from $ir$
  to $\id{B}$, we will talk of a \ndef{strong right deformation retract}.
\end{paragraph}

\begin{proposition}\label{prop:fib_prod_retr}
  Let
  \[
   (i_k : A_k \to B_k, r_k : B_k \to A_k, h_k : \Deltan{1} \times B_k \to
   B_k),
  \]
  for $k = 0,1,2$, be three left \resp{right} strong deformation retracts.
  If $f_0, f_1, g_0, g_1$ are morphisms such that the diagrams
  \[
    \xymatrix@C=3pc{
      A_0 \ar[r]^{f_0} \ar[d]_{i_0} & A_2 \ar[d]_{i_2} & A_1 \ar[l]_{f_1}
      \ar[d]^{i_1}
      \\
      B_0 \ar[r]_{g_0} & B_2 & B_1 \ar[l]^{g_1}
    }
  \]
  and
  \[
    \xymatrix@C=3pc{
      \Deltan{1} \times B_0 \ar[d]_{h_0} \ar[r]^-{\Deltan{1} \times g_0}
      &
      \Deltan{1} \times B_2 \ar[d]_{h_2}
      &
      \Deltan{1} \times B_1 \ar[d]^{h_1} \ar[l]_-{\Deltan{1} \times g_1}
      \\
      B_0 \ar[r]_{g_0} & B_2 & B_1 \ar[l]^{g_1}
    }
  \]
  commute, then
  \[
    \begin{split}
    \big(&i_0 \times_{i_2} i_1 : A_0 \times_{A_2} A_1 \to B_0 \times_{B_2} B_1,
    \\
     & r_0 \times_{r_2} r_1 : B_0 \times_{B_2} B_1 \to  A_0 \times_{A_2} A_1,
     \\
     & h_0 \times_{h_2} h_1 : \Deltan{1} \times (B_0 \times_{B_2} B_1) \to
      B_0 \times_{B_2} B_1\big)
    \end{split}
  \]
  is a strong left \resp{right} deformation retract.
\end{proposition}

\begin{proof}
  We only need to check that $r_0 \times_{r_2} r_1$ is well-defined; the
  fact that the triple of the statement is a strong left \resp{right}
  deformation retract will then follow by functoriality of the fiber
  product. Evaluating the second diagram of the statement at~$0$
  \resp{at~$1$} in $\Deltan{1}$, we get that the diagram
  \[
    \xymatrix@C=3pc{
      B_0 \ar[d]_{i_0r_0} \ar[r]^-{g_0}
      &
      B_2 \ar[d]_{i_2r_2}
      &
      B_1 \ar[d]^{i_1r_1} \ar[l]_-{g_1}
      \\
      B_0 \ar[r]_{g_0} & B_2 & B_1 \ar[l]^{g_1}
    }
  \]
  commutes. The fact that $i_2$ is a monomorphism and that the first
  diagram of the statement commutes implies then that the diagram
  \[
    \xymatrix@C=3pc{
      B_0 \ar[r]^{g_0} \ar[d]_{r_0} & B_2 \ar[d]_{r_2} & B_1 \ar[l]_{g_1}
      \ar[d]^{r_1}
      \\
      A_0 \ar[r]_{f_0} & A_2 & A_1 \ar[l]^{f_1}
    }
  \]
  commutes as well, thereby ending the proof.
\end{proof}

\section{Preliminaries on oplax transformations and comma \pdfoo-categories}

\begin{paragraph}\label{paragr:notation}
  We will denote by $\ooCat$ the category of strict \oo-categories and
  strict \oo-functors. All the \oo-categories and \oo-functors considered in
  this paper will be strict, and we will drop the adjective ``strict'' from
  now on.

  If $C$ is an \oo-category, we will denote by $C^\o$ the \oo-category
  obtained from $C$ by reversing all the $i$-cells for $i > 0$.

  We will denote by $\Dn{0}$ the terminal \oo-category and by $\Dn{1}$ the
  \oo-category associated to the category defined by the ordered set
  $\Deltan{1} = \{0 < 1\}$. We have two \oo-functors $\sigma, \tau : \Dn{0}
  \to \Dn{1}$ corresponding respectively to the objects $0$ and $1$ of
  $\Dn{1}$ and a unique \oo-functor $\kappa : \Dn{1} \to \Dn{0}$.
\end{paragraph}

We begin the section with some preliminaries on the Gray tensor product and
oplax transformations.

\begin{paragraph}
  The category $\ooCat$ is endowed with a biclosed monoidal category structure
  given by the so-called \ndef{Gray tensor product}, first introduced
  in~\cite{AlAglSteiner}. This tensor product is a generalization of the
  tensor product of $2$-categories introduced by Gray in~\cite{GrayFCT}. We
  will not need its precise definition and we will only recall the
  properties we will need. We refer the reader to \cite[Appendix
  A]{AraMaltsiJoint} for a comprehensive presentation in the spirit of our
  paper. If $A$ and $B$ are two \oo-categories, their (Gray) tensor product
  will be denoted by $A \otimes B$. For instance, one has
  \[
    \shorthandoff{;:}
    \Dn{1} \otimes \Dn{1} \simeq
    \raisebox{1.6pc}{
    $\xymatrix@C=2.5pc@R=2.5pc{
      \labelstyle (0,0)
      \ar[d]
      \ar[r]
      &
      \labelstyle (0, 1)
      \ar[d]
      \\
      \labelstyle (1, 0)
      \ar[r]
      &
      \labelstyle (1, 1)
      \ar@{}[u];[l]_(.30){}="x"
      \ar@{}[u];[l]_(.70){}="y"
      \ar@2"x";"y"
    }$
    }.
  \]
  The unit of this tensor product is
  the terminal \oo-category $\Dn{0}$.  The right and left internal $\Hom$
  will be denoted by $\HomOpLax$ and $\HomLax$ respectively, so that we have
  bijections
  \[
      \Hom_{\ooCat}(A \otimes B, C)
      \simeq \Hom_{\ooCat}(A, \HomOpLax(B, C))
  \]
  and
  \[
      \Hom_{\ooCat}(A \otimes B, C)
      \simeq \Hom_{\ooCat}(B, \HomLax(A, C)),
  \]
  natural in $A$, $B$ and $C$ in $\ooCat$.
\end{paragraph}

\begin{remark}
  The orientation of the non-trivial $2$-cell of $\Dn{1} \otimes \Dn{1}$
  in the diagram of the previous paragraph
  reveals that the Gray tensor product we work with is what we would call
  the \emph{oplax} Gray tensor product, as opposed to the \emph{lax} Gray
  tensor product in which this $2$-cell would be reversed.
\end{remark}

\begin{paragraph}\label{paragr:def_trans}
  Let $A$ and $B$ be two \oo-categories. The objects of the \oo-categories
  \[ \HomOpLax(A, B) \quadand \HomLax(A, B) \]
  are in canonical bijection with the \oo-functors from $A$ to $B$. If $u, v : A
  \to B$ are two such \oo-functors, a $1$-cell $\alpha$ of $\HomOpLax(A, B)$
  from $u$ to $v$ is called an \ndef{oplax transformation} from $u$ to $v$.
  We will then write $\alpha : u \tod v$. By definition, an oplax
  transformation $\alpha$ corresponds to an \oo-functor $\Dn{1} \to \HomOpLax(A, B)$
  and so, by adjunction, to an \oo-functor $\Dn{1} \otimes A \to B$. The
  fact that $\alpha$ has $u$ as source and $v$ as target translates as the
  commutativity of the diagram
  \[
    \xymatrix@C=3pc{
    A \ar[dr]^u \ar[d]_{\sigma \otimes A} \\
    \Dn{1} \otimes A \ar[r]^\alpha & B \\
    A \ar[ur]_v \ar[u]^{\tau \otimes A} & \pbox{,} \\
    }
  \]
  where we identify $A$ and $\Dn{0} \otimes A$. Alternatively, again by
  adjunction, an oplax transformation corresponds to an \oo-functor $A \to
  \HomLax(\Dn{1}, B)$. The source and the target of an oplax transformation
  given by such an \oo-functor are obtained by postcomposing by
  \[
    \begin{split}
      \HomLax(\sigma, B) & : \HomLax(\Dn{1}, B) \to \HomLax(\Dn{0}, B)
      \simeq B,
      \\
      \HomLax(\tau, B) & : \HomLax(\Dn{1}, B) \to \HomLax(\Dn{0}, B) \simeq
      B,
    \end{split}
  \]
  respectively.

  Similarly, $1$-cells of $\HomLax(A, B)$ are called \ndef{lax
  transformations}.
\end{paragraph}

\goodbreak

\begin{paragraph}
  Let $u : A \to B$ be an \oo-functor. We define the \ndef{identity oplax
  transformation} $\id{u} : u \tod u$ to be the oplax transformation
  corresponding to the composite
  \[ \xymatrix{\Dn{1} \otimes A \ar[r]^-{\kappa \otimes A} & \Dn{0} \otimes
  A \ar[r]^-{\sim} & A \ar[r]^u & B \pbox{,}} \]
  where the middle arrow is the canonical isomorphism.

  Let $v : A \to B$ be a second \oo-functor and let $\alpha : u \tod v$ be an
  oplax transformation, seen as an \oo-functor $\Dn{1} \otimes A \to B$.
  If $w : B \to C$ is an \oo-functor, we get an oplax transformation $w
  \comp \alpha : wu \tod wv$ by composing
  \[ \xymatrix{\Dn{1} \otimes A \ar[r]^-{\alpha} & B \ar[r]^w & C \pbox{.}} \]
  Similarly, if $w : C \to A$ is an \oo-functor, we get an oplax
  transformation $\alpha \comp w : uw \tod vw$ by composing
   \[
     \xymatrix@C=2.2pc{\Dn{1} \otimes C \ar[r]^{\Dn{1} \otimes w\,\,} & \Dn{1} \otimes A
     \ar[r]^-\alpha & B \pbox{.}}
   \]

   Finally, let $w : A \to B$ be a third \oo-functor and let $\beta : v \tod w$ be a
   second oplax transformation. The composition of $1$-cells in the
   \oo-category $\HomOpLax(A, B)$ gives an oplax transformation that we will
   denote by $\beta\alpha$. We have $\beta\alpha : u \tod w$.
\end{paragraph}

\begin{paragraph}
  The \oo-categories, \oo-functors and oplax transformations, with the
  operations defined in the previous paragraph, form a
  sesquicategory (see~\cite[Section 2]{StreetCatStruct} for a definition) that
  we will denote by $\ooCatOpLax$ (but they do not form a $2$-category!).
  Similarly, the \oo-categories, \oo-functors and lax transformations form a
  sesquicategory that we will denote by~$\ooCatLax$.
\end{paragraph}

\begin{paragraph}
  Let $i : A \to B$ be an \oo-functor. The \ndef{structure of a left}
  (resp.~\ndef{right}) \ndef{oplax transformation retract} on $i$ consists
  of
  \begin{enumerate}
    \item a retraction $r : B \to A$ of $i$ (so that we have $ri =
    \id{A}$),
    \item an oplax transformation $\alpha$ from $ir$ to $\id{B}$ (resp.~from
    $\id{B}$ to $ir$).
  \end{enumerate}
  We will say that the structure is \ndef{strong} if $\alpha \comp i =
  \id{i}$ and that it is \ndef{above its source} if~$r \comp \alpha =
  \id{r}$.

  We will say that $(i, r, \alpha)$, or simply $i$, is a \ndef{left}
  (resp.~\ndef{right}) \ndef{oplax transformation retract} if $(r, \alpha)$
  is a structure of left (resp.~right) oplax transformation retract on $i$.
  Such a retract will be said to be \ndef{strong} or \ndef{above its source}
  according to the properties of the structure $(r, \alpha)$.

  All the notions introduced in this paragraph admit lax variants
  obtained by replacing the oplax transformation $\alpha$ by a lax
  transformation.
\end{paragraph}

\begin{proposition}\label{prop:pullback_retr}
  Let $i : A \to B$ be a strong left \resp{right} oplax transformation retract above
  its source with retraction $r : B \to A$. Then for every diagram of pullback
  squares
  \[
    \xymatrix{
      A' \pullbackcorner \ar[d]_{i'} \ar[r]^u & A \ar[d]^i \\
      B' \pullbackcorner \ar[r] \ar[d]_{r'} & B \ar[d]^r \\
      A' \ar[r]_{u} & A \pbox{,}
    }
  \]
  the \oo-functor $i'$ is a strong left \resp{right} oplax transformation
  retract above its source with retraction~$r'$.
\end{proposition}

\begin{proof}
  This is a particular case of \cite[Proposition 5.6]{AraMaltsiThmAII}.
\end{proof}

We now recall the basic definitions and some properties of the comma
construction for \oo-categories that we introduced with Maltsiniotis in
\cite{AraMaltsiThmAII}.

\begin{paragraph}\label{paragr:def_comma}
  Let
  \[
    \xymatrix{
      A \ar[r]^u & C & B \ar[l]_v
    }
  \]
  be two \oo-functors. We define the \ndef{comma \oo-category}
  $u \comma_C v$, also simply denoted by~$u \comma v$, to be the iterated
  fiber product
  \[
    u \comma v = A \times_C \HomLax(\Dn{1}, C) \times_C B,
  \]
  projective limit of the diagram
  \[
    \xymatrix{
      A \ar[r]^u & C & \HomLax(\Dn{1}, C) \ar[l]_-{\pi_0} \ar[r]^-{\pi_1}
                 & C & B \ar[l]_v \pbox{,}
    }
  \]
  where $\pi_0 = \HomLax(\sigma, C)$ and $\pi_1 = \HomLax(\tau, C)$.

  In the case where $A = C$ and $u = \id{C}$, we will denote $u \comma v$ by
  $C \comma v$. Similarly, if $B = C$ and $v = \id{C}$, we will denote $u
  \comma v$ by $u \comma C$.

  The canonical projections induce \oo-functors
  \[
    \xymatrix{A & u \comma v \ar[l]_{p_1} \ar[r]^{p_2} & B}
  \]
  and, using the description of oplax transformations in terms of $\HomLax(\Dn{1}, C)$
  given in paragraph~\ref{paragr:def_trans}, an oplax transformation
  \[
    \shorthandoff{;}
    \xymatrix@C=1.5pc@R=1.5pc{
      & u \comma v \ar[dl]_{p_1} \ar[dr]^{p_2} \\
      A \ar[dr]_u \ar@{}[rr]_(.38){}="x"_(.62){}="y"
      \ar@2"x";"y"^{\kappa\,\,}
      & & B \ar[dl]^v \\
      & C & \pbox{.}
    }
  \]
  Moreover, the data of an \oo-functor $T \to u \comma v$ corresponds to
  the data of a diagram
  \[
    \shorthandoff{;}
    
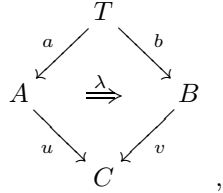
\xymatrix@C=1.5pc@R=1.5pc{
      & T \ar[dl]_a \ar[dr]^b \\
      A \ar[dr]_u \ar@{}[rr]_(.35){}="x"_(.65){}="y"
      \ar@2"x";"y"^{\lambda\,\,}
      & & B \ar[dl]^v \\
      & C & \pbox{,}
    }
  \]
  where $a$ and $b$ are \oo-functors and $\lambda : ua \tod vb$ is an oplax
  transformation.
\end{paragraph}

\begin{paragraph}\label{paragr:def_comma_morph}
  Fix $v : B \to C$ an \oo-functor and consider a diagram
  \[
    \shorthandoff{;}
    
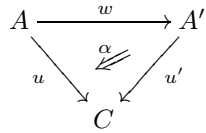
\xymatrix@C=1.5pc{
      A \ar[rr]^w \ar[dr]_(0.40){\phantom{u'}u}_(.60){}="f" & & A' \ar[dl]^(0.40){u'} \\
      & C
      \ar@{}"f";[ur]_(.15){}="ff"
      \ar@{}"f";[ur]_(.55){}="oo"
      \ar@<-0.0ex>@2"oo";"ff"_\alpha
      &
    }
  \]
  in $\ooCat$, where $\alpha : u'w \tod u$ is an oplax transformation. We
  define an \oo-functor
  \[ (w, \alpha) \comma v : u \comma v \to u' \comma v \]
  in the following way. Let
  \[
    \shorthandoff{;}
    
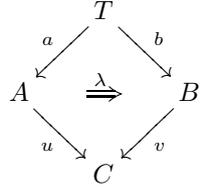
\xymatrix@C=1.5pc@R=1.5pc{
      & T \ar[dl]_a \ar[dr]^b \\
      A \ar[dr]_u \ar@{}[rr]_(.35){}="x"_(.65){}="y"
      \ar@2"x";"y"^{\lambda\,\,}
      & & B \ar[dl]^v \\
      & C &
    }
  \]
  be a diagram corresponding to an \oo-functor $T \to u \comma v$. By
  composing the diagram
  \[
    \shorthandoff{;}
    
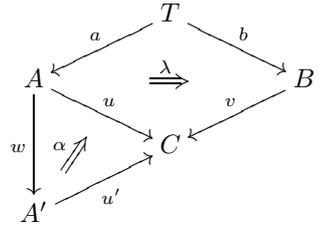
\xymatrix@R=1pc@C=3pc{
      & T \ar[dl]_a \ar[dr]^b \\
      A \ar[dd]_w \ar[dr]^u_{}="f" & & B \ar[dl]_v_{}="g"
      \ar@{}[ll];[]_(0.40){}="x"_(0.60){}="y"
      \ar@2"x";"y"^{\lambda\,\,}
       \\
        & C \\
      A' \ar[ur]_{u'}
      & &
      \ar@{}[ll];"f"_(0.35){}="sa"_(0.85){}="ta"
      \ar@2"sa";"ta"^{\alpha}
      \pbox{,}
    }
  \]
  we get a diagram
  \[
    \shorthandoff{;}
    
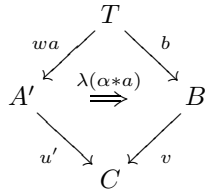
\xymatrix@C=1.5pc@R=1.5pc{
      & T \ar[dl]_{wa} \ar[dr]^b \\
      A' \ar[dr]_{u'} \ar@{}[rr]_(.35){}="x"_(.65){}="y"
      \ar@2"x";"y"^{\lambda(\alpha \comp a)} 
      & & B \ar[dl]^{v\phantom{u'}} \\
      & C &
    }
  \]
  corresponding to an \oo-functor $T \to u' \comma v$. This correspondence
  is natural in $T$ and defines, by the Yoneda lemma, our \oo-functor $(w,
  \alpha) \comma v : u \comma v \to u' \comma v$.

  One checks (see \cite[Proposition 6.9]{AraMaltsiThmAII}) that the square
  and the triangle
  \[
    \xymatrix@C=3pc{
      u \comma v \ar[r]^{(w, \alpha) \comma v\,\,} \ar[d]_{p_1} &
      u' \comma v \ar[d]^{p_1}
      \\
      A \ar[r]_w & A'
    }
    \qquad
    \xymatrix@C=1.5pc{
      u \comma v \ar[rr]^{(w, \alpha) \comma v} \ar[dr]_{p_2} & &
      u' \comma v \ar[dl]^{p_2}
      \\
      & B
    }
  \]
  are commutative.

  Note that in the case where $\alpha$ is the identity,
  so that the original triangle is commutative, the diagram corresponding to
  the \oo-functor $T \to u' \comma v$ becomes
  \[
    \shorthandoff{;}
    \xymatrix@C=1.5pc@R=1.5pc{
      & T \ar[dl]_{wa} \ar[dr]^b \\
      A' \ar[dr]_{u'} \ar@{}[rr]_(.35){}="x"_(.65){}="y"
      \ar@2"x";"y"^{\lambda\,\,}
      & & B \ar[dl]^{v\phantom{u'}} \\
      & C & \pbox{,}
    }
  \]
  showing that
  \[ (w, \id{u}) \comma v : u \comma v \to u' \comma v \]
  is nothing but the \oo-functor
  \[
    w \times_C \HomLax(\Dn{1}, C) \times_C B :
    A \times_C \HomLax(\Dn{1}, C) \times_C B \to
    A' \times_C \HomLax(\Dn{1}, C) \times_C B.
  \]

  If now $u : A \to C$ is an \oo-functor and
  \[
    \shorthandoff{;}
    \xymatrix@C=1.5pc{
      B \ar[rr]^w \ar[dr]_(0.40){\phantom{v'}v}_(.60){}="f" & & B' \ar[dl]^(0.40){v'} \\
      & C
      \ar@{}"f";[ur]_(.15){}="ff"
      \ar@{}"f";[ur]_(.55){}="oo"
      \ar@<-0.0ex>@{<=}"oo";"ff"_\beta
      &
    }
  \]
  is a diagram in $\ooCat$, where $\beta : v \tod v'w$ is an oplax
  transformation, we define similarly an \oo-functor
  \[ u \comma (\beta, w) : u \comma v \to u \comma v' \]
  enjoying analogous properties.
\end{paragraph}

\begin{remark}
  We proved with Maltsiniotis in \cite[Appendix B]{AraMaltsiThmAII} that, if
  $C$ is an \oo-category, the comma construction actually defines a functor
  and even a sesqui\-functor
  \[ \var \comma_C \var : \tr{\ooCatOpLaxGray}{C} \times \trto{\ooCatOpLaxGray}{C}
  \to \ooCatOpLax, \]
  where $\tr{\ooCatOpLaxGray}{C}$ and $\trto{\ooCatOpLaxGray}{C}$ are some
  appropriate sesquicategories. We will only need some consequences of this
  result which we will recall in this section.
\end{remark}

\begin{proposition}\label{prop:comma_fib_prod}
  Let
  \[
    \xymatrix{
      A \ar[r]^u & C & B \ar[l]_v
    }
  \]
  be two \oo-functors. We have pullback squares
  \[
    \xymatrix@C=3.5pc{
      u \comma v \pullbackcorner \ar[d]_{p_1} \ar[r]^-{(u, \id{u}) \comma v} &
      C \comma v \ar[d]^{p_1} \\
      A \ar[r]_u & C
    }
    \qquad
    \xymatrix@C=3.5pc{
      u \comma v \pullbackcorner \ar[d]_{u \comma (\id{v}, v)} \ar[r]^-{p_2} &
      B \ar[d]^{v} \\
      u \comma C \ar[r]_-{p_2} & C \pbox{.}
    }
  \]
\end{proposition}

\begin{proof}
  We will only treat the first pullback square, the proof for the second one
  being similar. We have
  \[
    \begin{split}
      u \comma v & = A \times_C \HomLax(\Dn{1}, C) \times_C B \\
       & \simeq A \times_C \big(C \times_C \HomLax(\Dn{1}, C) \times_C B\big) \\
       & = A \times_C (C \comma v),
    \end{split}
  \]
  where the isomorphism is induced by the \oo-functors $p_1$ and $u \times_C
  \HomLax(\Dn{1}, C) \times_C B$. But by
  paragraph~\ref{paragr:def_comma_morph}, the latter \oo-functor is nothing
  but $(u, \id{u}) \comma v$, thereby proving the result.
\end{proof}

\begin{proposition}\label{prop:comma_retr}
  Let
    \[
      \xymatrix{
        A \ar[r]^u & C & B \ar[l]_v
      }
    \]
  be two \oo-functors.
  \begin{enumerate}
    \item\label{item:comma_retr_a} If $i : A' \to A$ is a strong left oplax
      transformation retract, then so is
      \[ (i, \id{ui}) \comma v :  (ui) \comma v \to u \comma v. \]
      More precisely, if $(r, \alpha)$ is a structure of strong left oplax
      transformation retract on $i$, then there exists a structure of strong
      left oplax transformation retract on $(i, \id{ui}) \comma v$ of the
      form~$(r', \gamma)$
      % $((r, u \comp \alpha) \comma u, \gamma)$,
      with $\gamma$ compatible with $\alpha$ in the sense that $p_1 \comp
      \gamma = \alpha \comp p_1$, where~$p_1 : u \comma v \to A$.

    \item\label{item:comma_retr_b} If $j : B' \to B$ is a strong right oplax
      transformation retract, then so is
      \[ u \comma (\id{vj}, j) :  u \comma (vj) \to u \comma v. \]
      More precisely, if $(r, \beta)$ is a structure of strong right oplax
      transformation retract on $j$, then there exists a structure of strong
      right oplax transformation retract on $u \comma (\id{vj}, j)$ of the
      form~$(r', \gamma)$
      % $(u \comma (v \comp \beta, r), \gamma)$,
      with $\gamma$ compatible with $\beta$ in the sense that $p_2 \comp
      \gamma = \beta \comp p_2$, where~$p_2 : u \comma v \to B$.
  \end{enumerate}
\end{proposition}

\begin{proof}
  The first assertion of \ref{item:comma_retr_a} is exactly~\cite[Corollary
  B.2.8.(a)]{AraMaltsiThmAII}. The proof of this corollary actually produces
  a structure $(r', \gamma)$ and the fact that this
  $\gamma$ is compatible with $\alpha$ follows from \cite[Proposition
  B.2.9]{AraMaltsiThmAII}. The situation is similar for~\ref{item:comma_retr_b}.
\end{proof}

We now introduce slice \oo-categories in terms of comma \oo-categories. For
a more concrete description, we refer the reader to
\cite[paragraph~4.1]{AraMaltsiThmAI} (see
\cite[Proposition~7.1]{AraMaltsiThmAII} for the comparison of the two
definitions).

\begin{paragraph}
  If $A$ is an \oo-category and $a$ is an object of $A$, we will denote by
  $\cotr{A}{a}$ the \oo-category
  \[ \cotr{A}{a} = a \comma A, \]
  where $a$ is seen as an \oo-functor $\Dn{0} \to A$. The \oo-functor
  $p_2 : \cotr{A}{a} \to A$ will be called the \ndef{forgetful \oo-functor}.

  More generally, if $u : A \to B$ is an \oo-functor and $b$ is an object
  of $B$, we set
  \[ \cotr{A}{b} = b \comma u, \]
  where $b$ is also seen as an \oo-functor $\Dn{0} \to B$. It
  follows from Proposition~\ref{prop:comma_fib_prod} that we have
  \[ \cotr{A}{b} = \cotr{B}{b} \times_B A, \]
  where the fiber product involves the forgetful \oo-functor $\cotr{B}{b}
  \to B$ and $u$. In this setting, we also have a \ndef{forgetful
  \oo-functor} $\cotr{A}{b} \to A$.

  If $f : b \to b'$ is a $1$-cell of $B$, we define the \oo-functor
  \[ \cotr{A}{f} : \cotr{A}{b'} \to \cotr{A}{b} \]
  to be
  \[ (\id{\Dn{0}}, f) \comma u : b' \comma u \to b \comma u, \]
  where $f$ is seen as an oplax transformation
  \[
    \shorthandoff{;}
    \xymatrix@C=1.5pc{
      \Dn{0} \ar[rr]^{\id{\Dn{0}}} \ar[dr]_(0.40){b'}_(.60){}="f" & & \Dn{0}
      \ar[dl]^(0.40){b\phantom{b'}} \\
      & C
      \ar@{}"f";[ur]_(.15){}="ff"
      \ar@{}"f";[ur]_(.55){}="oo"
      \ar@<-0.0ex>@2"oo";"ff"_f
      & \pbox{.}
    }
  \]

\end{paragraph}

\begin{proposition}\label{prop:comma_triangle}
  Let
  \[
    \xymatrix{
      A \ar[r]^u & C & B \ar[l]_v
    }
  \]
  be two \oo-functors and let $f : a \to a'$ be a $1$-cell of $A$. Then
  there exists an oplax transformation
  \[
    \shorthandoff{;}
    \xymatrix@C=1.5pc{
      \cotr{B}{u(a')} \ar[rr]^{\cotr{B}{u(f)}}
      \ar[dr]_(0.40){(a', \id{u(a')}) \comma v}_(.60){}="f" & &
    \cotr{B}{u(a)} \ar[dl]^(0.40){(a, \id{u(a)}) \comma v} \\
      & u \comma v
      \ar@{}"f";[ur]_(.15){}="ff"
      \ar@{}"f";[ur]_(.55){}="oo"
      \ar@<-0.0ex>@2"oo";"ff"
      & \pbox{.}
    }
  \]
\end{proposition}

\begin{proof}
  We will use the fact that the comma construction $\var \comma v$ extends
  to a sesquifunctor from a certain sesquicategory $\tr{\ooCatOpLaxGray}{C}$
  to the sesquicategory $\ooCatOpLax$ (see
  \cite[Theorem~B.2.6]{AraMaltsiThmAII}). All we need to know about the
  source sesquicategory of this extension (see \cite[paragraph
  B.1.18]{AraMaltsiThmAII} for a complete description) is that the
  diagram
   \[
      \shorthandoff{;:}
      \xymatrix@C=1.5pc@R=3.4pc{
        \Dn{0} \ar@/^3ex/[rr]^{a}_{}="1" \ar@/_2ex/[rr]_{a'}_{}="0"
        \ar[dr]_(.60){}="2"_{u(a')}
        \ar@{}"1";"0"_(.10){}="1'"_(.90){}="0'"
        \ar@2"1'";"0'"^{\,f}
        & & A \ar[dl]^{u\phantom{u(a')}}_(.60){}="3" \\
        & C
        \ar@{}"2";"3"^{\textstyle =}
        &
        }
        \qquad
        \raisebox{-1.3pc}{\xymatrix@C=1pc{\ar@3[r]^{=\,\,\,} &}}
        \qquad
      \raisebox{1.2pc}{
      $\xymatrix@C=1.5pc@R=0.05pc{
        & \Dn{0} \ar@/^1ex/[dr]^a
        \ar[dddddddd]^(.20){u(a)}_(0.35){}="1"_(.47){}="2" \\
        \Dn{0} \ar@{=}@/^1ex/[ur]
        \ar[dddddddr]_(.40){}="0"_{u(a')}
        & & A \ar[dddddddl]^{u\phantom{u(a')}}_(.47){}="3"
        \\ \\ \\ \\ \\ \\ \\
        & C
        \ar@{}"1";"0"_(.10){}="1'"_(.90){}="0'"
        \ar@2"1'";"0'"_{u(f)}
        \ar@{}"2";"3"^{\textstyle =}
        & \pbox{,}
        }$}
  \]
  where the $2$-arrows denote oplax transformations and the $3$-arrow is
  formally an identity ``oplax $2$-transformation'' but can be interpreted
  simply as an equality of oplax transformations,
  defines a $2$-cell from the composite of $1$-cells associated to the
  $2$\nbd-triangles
  \[
      \shorthandoff{;}
      \xymatrix@C=1.5pc@R=0.05pc{
        & \Dn{0}
        \ar[dddddddd]^{u(a)}_(0.35){}="1"_(.47){}="2" \\
        \Dn{0} \ar@{=}@/^1ex/[ur]
        \ar[dddddddr]_(.40){}="0"_{u(a')}
        & &
        \\ \\ \\ \\ \\ \\ \\
        & C
        \ar@{}"1";"0"_(.10){}="1'"_(.90){}="0'"
        \ar@2"1'";"0'"_{u(f)}
        &
        }
      \qquad
      \xymatrix@C=1.5pc@R=0.05pc{
        & \Dn{0} \ar@/^1ex/[dr]^a
        \ar[dddddddd]_{u(a)}_(0.35){}="1"_(.47){}="2" \\
        & & A \ar[dddddddl]^{u\phantom{u(a')}}_(.47){}="3"
        \\ \\ \\ \\ \\ \\ \\
        & C
        \ar@{}"2";"3"^{\textstyle =}
        &
        }
  \]
  to a $1$-cell associated to the $2$-triangle
   \[
      \shorthandoff{;:}
      \xymatrix@C=1.5pc@R=3.4pc{
        \Dn{0} \ar@/_2ex/[rr]^{a'}_{}="0"
        \ar[dr]_(.60){}="2"_{u(a')}
        & & A \ar[dl]^{u\phantom{u(a')}}_(.60){}="3" \\
        & C
        \ar@{}"2";"3"^{\textstyle =}
        & \pbox{.}
        }
    \]
  (In the notation of \cite[paragraph B.1.18]{AraMaltsiThmAII}, this
  $2$-cell is denoted by $(f, \id{u(f)})$ and we have
    \[
     (f, \id{u(f)}) :
      (a, \id{u(a)}) \comp_0 (\id{\Dn{0}}, u(f)) \Rightarrow (a', \id{u(a')})
    \]
  in $\tr{\ooCatOpLaxGray}{C}$.)
  By applying the sesquifunctoriality of $\var \comma v$ to this $2$-cell, we get
  a $2$\nbd-cell in $\ooCatOpLax$, that is, an oplax transformation, from the
  composite of the \oo-functors $(a, \id{u(a)}) \comma v$ and $(\id{\Dn{0}},
  u(f)) \comma v = \cotr{B}{u(f)}$ to the \oo-functor $(a', \id{u(a')})
  \comma v$, thereby proving the result.
\end{proof}

\goodbreak

We end the section with a kind of mapping space factorization for
\oo-functors, involving comma \oo-categories, that will be needed in our
proof of the \oo-categorical Theorem B.

\begin{paragraph}\label{paragr:fact}
  Let $u : A \to B$ be an \oo-functor. We will see that $u$ factors as
  \[
    \xymatrix{
      A \ar[r]^-j & B \comma u \ar[r]^-{p_1} & B \pbox{,}
    }
  \]
  for some \oo-functor $j$. As the composition
  \[ \xymatrix{\Dn{0} \ar[r]^\tau & \Dn{1} \ar[r]^\kappa & \Dn{0}} \]
  is the identity, by applying the functor $\HomLax(\var, B)$, we get a
  factorization
  \[ \xymatrix{B \ar[r]^-\iota & \HomLax(\Dn{1}, B) \ar[r]^-{\pi_1} & B} \]
  of the identity of $B$. By pulling back this factorization along $u$, we
  get a diagram of pullback squares
  \[
    \xymatrix{
      A \pullbackcorner \ar[d]_j \ar[r]^u & B \ar[d]^\iota \\
      B \comma u \pullbackcorner \ar[r] \ar[d]_{p_2} & \HomLax(\Dn{1}, B)
      \ar[d]^{\pi_1} \\
      A \ar[r]_{u} & B \pbox{,}
    }
  \]
  defining our \oo-functor $j$. The equality $\pi_0\iota = \id{B}$ easily
  implies that we have $u = p_1j$, as announced.
\end{paragraph}

\begin{proposition}\label{prop:j_retr}
  The \oo-functor $j : A \to B \comma u$ defined in the previous paragraph
  is a strong right oplax transformation retract above its source with
  retraction~$p_2 : B \comma u \to A$.
\end{proposition}

\begin{proof}
  Consider the diagram of pullback squares of the previous paragraph.
  As $1$ is a terminal object of the $1$-category $\Dn{1}$, the \oo-functor
  $\tau$ is a strong right lax transformation retract above its source with
  retraction $\kappa$. For formal reasons
  (see~\cite[Example~C.23.(f)]{AraMaltsiJoint}), the functor
  $\HomLax({-}, B)$ extends to a sesquifunctor $(\ooCatLax)^\op \to
  \ooCatOpLax$, where $\C^\op$, for $\C$ a sesquicategory, denotes the
  sesquicategory obtained from $\C$ by reversing the $1$-cells.
  This implies that $\iota = \HomLax(\kappa, B)$ is a strong right
  oplax transformation retract above its source with retraction~$\pi_1 =
  \HomLax(\tau, B)$. The result then follows from
  Proposition~\ref{prop:pullback_retr}.
\end{proof}

\begin{paragraph}\label{paragr:fact_dual}
  Similarly, any \oo-functor $u : A \to B$ factors as
  \[
    \xymatrix{
      A \ar[r]^-{j'} & u \comma B \ar[r]^-{p_2} & B \pbox{,}
    }
  \]
  where $j'$ is a strong \emph{left} oplax transformation retract above its source
  with retraction~$p_1 : u \comma B \to A$.
  This can be proven either by adapting the previous proof or by a duality
  argument involving the automorphism $C \mapsto C^\o$ of $\ooCat$ (see
  paragraph~\ref{paragr:notation}).
\end{paragraph}

\section{A Quillen Theorem B for \pdfoo-categories}

\begin{paragraph}
  We will denote by $N : \ooCat \to \pref{\cDelta}$ the so-called
  \ndef{Street nerve}, introduced by Street in~\cite{StreetOrient}. We will briefly
  recall in Appendix~\ref{app:contr_orient} (see
  paragraph~\ref{paragr:def_orientals}) one of its definition using Steiner's
  theory~\cite{Steiner}. This definition is not needed in this section and
  we will recall all the properties we will use.

  This nerve functor is induced by a cosimplicial object $\cO : \cDelta \to
  \ooCat$ sending $\Deltan{n}$ to the so-called \ndef{$n$-th oriental}
  $\On{n}$. Here are pictures of orientals in low dimension:
  \[
    \shorthandoff{;}
    \On{0} = \Dn{0} = \xymatrix{\{0\}}, \qquad
    \On{1} = \Dn{1} = \xymatrix{0 \ar[r] & 1},
    \qquad
    \On{2} =
    \raisebox{1.5pc}{
    $\xymatrix{
      & 2
      \\
      0 \ar[r] \ar[ur]_{}="s" & 1 \ar[u]
      \ar@{}"s";[]_(0.05){}="ss"_(0.85){}="tt"
      \ar@2"ss";"tt"
    }$
    }
    \text{,}
  \]
  \[
    \shorthandoff{;}
    \On{3} =
    \raisebox{1.5pc}{
    $\xymatrix{
      0 \ar[r]_(0.60){}="03" \ar[d] \ar[dr]_{}="02"_(0.60){}="02'" &
      3
      &
      0 \ar[r]_(0.40){}="03'" \ar[d]_{}="t3"
        &
      3
      \\
      1 \ar[r] & 2 \ar[u]_{}="s3"
      &
      1 \ar[r] \ar[ur]_{}="13"_(0.40){}="13'" & 2 \ar[u]
      \ar@{}"s3";"t3"_(0.20){}="ss3"_(0.80){}="tt3"
      \ar@3"ss3";"tt3"
      \ar@{}"13";[]_(0.05){}="s123"_(0.85){}="t123"
      \ar@2"s123";"t123"
      \ar@{}"02";[lll]_(0.05){}="s012"_(0.85){}="t012"
      \ar@2"s012";"t012"
      \ar@{}"03'";"13'"_(0.05){}="s013"_(0.85){}="t013"
      \ar@2"s013";"t013"
      \ar@{}"03";"02'"_(0.05){}="s023"_(0.85){}="t023"
      \ar@2"s023";"t023"
      \\
    }$
    }
    \text{.}
  \]
  By definition, if $C$ is an \oo-category, we have $(NC)_p =
  \Hom_{\ooCat}(\On{p}, C)$. When $C$ is a $1$-category, then $NC$ coincides
  with the classical nerve functor. As Street's nerve is induced by a
  cosimplicial object, it admits as a left adjoint the Kan extension of this
  cosimplicial object along the Yoneda embedding. In particular, it
  preserves limits.
\end{paragraph}

\begin{paragraph}
  We will say that an \oo-functor $u : A \to B$ is a \ndef{Thomason
  equivalence} if its Street's nerve $Nu : NA \to NB$ is a simplicial weak
  equivalence.
\end{paragraph}

\begin{paragraph}\label{paragr:nerve_trans}
  Let $u, v : A \to B$ be two \oo-functors and let $\alpha : u \tod v$ be an
  oplax transformation. We constructed in
  \cite[Appendix~A]{AraMaltsiThmAII}, with Maltsiniotis, a simplicial
  homotopy~$N\alpha$ from $Nu$ to $Nv$. We will briefly recall the
  definition of $N\alpha$ in Appendix~\ref{app:contr_orient} (see
  paragraph~\ref{paragr:desc_Nalpha}) but all we will need about $N \alpha$ in
  this section is the following proposition.
\end{paragraph}

\begin{proposition}\label{prop:nerve_sesqui}
  Let $u, v : A \to B$ be two \oo-functors and let $\alpha : u \tod v$ be an
  oplax transformation.
  \begin{enumerate}
    \item If $w : B \to C$ is an \oo-functor, then we have
     \[ N(w \comp \alpha) = N(w)N(\alpha). \]
    \item If $w : C \to A$ is an \oo-functor, then we have
      \[ N(\alpha \comp w) = N(\alpha)(\Delta_1 \times N(w)). \]
  \end{enumerate}
\end{proposition}

\begin{proof}
  This is \cite[Proposition A.14]{AraMaltsiThmAII}.
\end{proof}

\begin{proposition}\label{prop:retr_Thomason}
  If $(i : A \to B, r, \alpha)$ is a strong left \resp{right} oplax transformation
  retract, then $(Ni : NA \to NB, Nr, N\alpha)$ is a strong left
  \resp{right} deformation retract and in particular $i$ and $r$ are
  Thomason equivalences.
\end{proposition}

\begin{proof}
  This follows from paragraph~\ref{paragr:nerve_trans} and the previous
  proposition.
\end{proof}

\begin{paragraph}
  We will say that a commutative square
  \[
    \xymatrix{
      A \ar[d]_{u} \ar[r]^v & A' \ar[d]^{u'} \\
      B \ar[r]_w & B'
    }
  \]
  in $\ooCat$ is \ndef{a homotopy pullback square} if the commutative square
  \[
    \xymatrix{
      NA \ar[d]_{Nu} \ar[r]^{Nv} & NA' \ar[d]^{Nu'} \\
      NB \ar[r]_{Nw} & NB'
    }
  \]
  of simplicial sets is a homotopy pullback square (as in
  paragraph~\ref{paragr:simpl_we}). Homotopy pullback squares in $\ooCat$
  inherit many properties of homotopy pullback squares in simplicial sets:
  for instance they compose, and a square as above in which $u$ and $u'$ are
  both Thomason equivalences is a homotopy pullback square.
\end{paragraph}

\begin{remark}
  One can show that a commutative square in $\ooCat$ is a homotopy pullback
  square in the sense of the previous paragraph if and only if it induces a
  pullback square in the weak $(\infty, 1)$-category obtained from $\ooCat$
  by weakly inverting Thomason equivalences. This follows from (a mild
  generalization) of~\hbox{\cite[Theorem~5.6]{Gagna}}.
\end{remark}

\goodbreak

We now introduce the notion corresponding to the hypothesis of Theorem B.

\begin{paragraph}
  Let $u : A \to B$ be an \oo-functor. We will say that $u$ is
  \ndef{colocally homotopically constant} if, for every $1$-cell $f : b \to
  b'$ of $B$, the \oo-functor $\cotr{A}{f} : \cotr{A}{b'} \to \cotr{A}{b}$
  is a Thomason equivalence.
\end{paragraph}

\goodbreak

The following proposition is the crucial step in our proof of the
\oo-categorical Theorem~B.

\begin{proposition}\label{prop:lemma}
  If $u : A \to B$ is a colocally homotopically constant \oo-functor, then
  any pullback square
  \[
    \xymatrix{
      C \pullbackcorner \ar[d] \ar[r] & B \comma u \ar[d]^{p_1} \\
      D \ar[r] & B
    }
  \]
  is a homotopy pullback square.
\end{proposition}

\begin{proof}
  Since the nerve functor preserves fiber products, by
  Proposition~\ref{prop:simpl_thmB}, it suffices to show that $N(p_1)$
  satisfies condition~\ref{cond:loc_const_zero} of this proposition. So consider a
  diagram of pullback squares of the form
  \[
    \xymatrix{
      Q_i \pullbackcorner \ar[r]^{f_i} \ar[d] & P \pullbackcorner \ar[d] \ar[r] &
        N(B \comma u) \ar[d]^{N(p_1)} \\
        \Deltan{0} \ar[r]_i & \Deltan{m} \ar[r]_x & NB \pbox{.}
    }
  \]
  By Proposition~\ref{prop:comma_fib_prod}, we also have a diagram of
  pullback squares
  \[
    \xymatrix{
      x(i) \comma u \pullbackcorner \ar[r]^{i_\ast} \ar[d]_{p_1} & x \comma u
      \pullbackcorner \ar[d]_{p_1} \ar[r] &
        B \comma u \ar[d]^{p_1} \\
        \On{0} \ar[r]_i & \On{m} \ar[r]_x & B \pbox{,}
    }
  \]
  where $i_\ast = (i, \id{x(i)}) \comma u$. Using again the fact that the
  nerve functor preserves fiber products, we get a canonical isomorphism
  between $Q_i$ and $N(x(i) \comma u)$. We thus have to
  show that
 \[ f_i : N(x(i) \comma u) \to P \] 
  is a weak equivalence.

  Denote by $\eta : \Deltan{m} \to N(\On{m})$ the adjunction
  morphism. By one of the triangular identities, the composite
  \[
    \xymatrix{\Deltan{m} \ar[r]^-\eta & N(\On{m}) \ar[r]^-{Nx} & NB}
  \]
  is $x : \Deltan{m} \to NB$ and we get a diagram of pullback squares
  \[
    \xymatrix{
      N(x(i) \comma u) \pullbackcorner \ar[r]^-{f_i} \ar[d]_{N(p_1)} &
      P \pullbackcorner \ar[r]^-g \ar[d] & N(x \comma u)
      \pullbackcorner \ar[d]_{N(p_1)} \ar[r] &
      N(B \comma u) \ar[d]^{N(p_1)} \\
      \Deltan{0} \ar[r]_-i &
      \Deltan{m} \ar[r]_-\eta & N(\On{m}) \ar[r]_-{Nx} & NB \pbox{.}
    }
  \]
  Note that we have $gf_i = N(i_\ast)$.

  To prove that $f_i$ is a weak equivalence, we proceed in three steps:
  \begin{enumerate}[wide,label=({\arabic*})]
    \item We show that $f_0$ is a weak equivalence. To do so,
  we will use Proposition~\ref{prop:fib_prod_retr} to prove that
  \[ f_0 : N(x(0) \comma u) \to P \]
  is a strong left deformation retract. To begin with, note that $f_0$ can
  be identified with the fiber product of the vertical maps of the
  commutative diagram
  \[
    \xymatrix@C=2.5pc{
      N(x(0) \comma u) \ar[r]^-{N(p_1)}
      \ar[d]_{N(0_\ast)}
      &
      N(\On{0})
      \ar[d]_{N(0)}
      &
      \Deltan{0} \ar[l]_-\eta \ar[d]^0
      \\
      N(x \comma u) \ar[r]_{N(p_1)}
      &
      N(\On{m})
      &
      \Deltan{m} \ar[l]^-\eta \pbox{,}
    }
  \]
  the left square being commutative by
  paragraph~\ref{paragr:def_comma_morph} and $N(\On{0})$ being isomorphic
  to~$\Deltan{0}$. The morphism $0 : \Deltan{0} \to \Deltan{m}$ is of course
  a strong left deformation retract, with a unique retraction and a unique
  simplicial homotopy $k$. We will prove in Appendix~\ref{app:contr_orient}
  (see in particular Propositions~\ref{prop:contr_orient}
  and~\ref{prop:trans_compat}) that there exists a structure of strong left
  oplax transformation retract $(r, \alpha)$ on $0 : \On{0} \to \On{m}$
  making the square
  \[
    \xymatrix@C=3pc{
      \Deltan{1} \times N(\On{m}) \ar[d]_{N\alpha}
      &
      \Deltan{1} \times \Deltan{m} \ar[l]_-{\Deltan{1} \times \eta} \ar[d]^{k}
      \\
      N(\On{m})
      &
      \Deltan{m} \ar[l]^-\eta
    }
  \]
  commute. In particular, by Proposition~\ref{prop:retr_Thomason},
  the morphism $N(0) : N(\On{0}) \to N(\On{m})$ is a strong left deformation
  retract with retraction $Nr$ and homotopy $N\alpha$.
  Proposition~\ref{prop:comma_retr} then implies that
  there exists a structure of strong left oplax transformation
  retract~$(r', \gamma)$ on $0_\ast : x(0) \comma u \to x \comma u$
  satisfying $p_1 \comp \gamma = \alpha \comp p_1$. In particular, again by
  Proposition~\ref{prop:retr_Thomason}, the morphism $N(0_\ast) : N(x(0)
  \comma u) \to N(x \comma u)$ is a strong left deformation retract with
  retraction~$Nr'$ and homotopy~$N\gamma$. By
  Proposition~\ref{prop:nerve_sesqui}, applying~$N$ to the equality $p_1
  \comp \gamma = \alpha \comp p_1$ gives the commutativity of the diagram
  \[
    \xymatrix@C=4pc{
      \Deltan{1} \times N(x \comma u) \ar[d]_{N\gamma}
      \ar[r]^-{\Deltan{1} \times N(p_1)}
      &
      \Deltan{1} \times N(\On{m}) \ar[d]^{N\alpha}
      \\
      N(x \comma u) \ar[r]_{N(p_1)}
      &
      N(\On{m}) \pbox{.}
    }
  \]
  We are thus in position to apply Proposition~\ref{prop:fib_prod_retr} and
  we get that $f_0$ is a strong left deformation retract and hence a weak
  equivalence.

  \item We show that $g$ is a weak equivalence. We proved in the
  previous step that $0_\ast$ is a strong left oplax transformation retract
  and $N(0_\ast)$ is thus a weak equivalence. As~$N(0_\ast) = gf_0$, this
  implies that $g$ is a weak equivalence.

  \item We show that $f_i$ is a weak equivalence. Let $l$ be any
  $1$-cell from $0$ to $i$ in $\On{m}$. Using
  Proposition~\ref{prop:comma_triangle}, we get an oplax transformation
  \[
    \shorthandoff{;}
    \xymatrix@C=1.5pc{
      \cotr{A}{x(i)} \ar[rr]^{\cotr{A}{x(l)}}
      \ar[dr]_(0.40){i_\ast}_(.60){}="f" & &
      \cotr{A}{x(0)} \ar[dl]^(0.40){0_\ast} \\
      & x \comma u
      \ar@{}"f";[ur]_(.15){}="ff"
      \ar@{}"f";[ur]_(.55){}="oo"
      \ar@<-0.0ex>@2"oo";"ff"
      & \pbox{.}
    }
  \]
  We already proved that $0_\ast$ is a Thomason equivalence and
  $\cotr{A}{x(l)}$ is a Thomason equivalence by hypothesis. Using
  paragraph~\ref{paragr:nerve_trans}, we get that $N(i_\ast)$ is homotopic
  to a weak equivalence and is thus a weak equivalence. The equality
  $N(i_\ast) = gf_i$ and the fact proven above that $g$ is a weak
  equivalence then implies that $f_i$ is a weak equivalence, thereby ending
  the proof. \qedhere
  \end{enumerate}
\end{proof}

\begin{theorem}\label{thm:comma_homot}
  Let
  \[
    \xymatrix{
      A \ar[r]^u & C & B \ar[l]_v
    }
  \]
  be two \oo-functors. If $v$ is colocally homotopically constant,
  then the pullback square
  \[
    \xymatrix{
      u \comma v \pullbackcorner \ar[d] \ar[r] & B \ar[d]^v \\
      u \comma C \ar[r]_{p_2} & C
    }
  \]
  is a homotopy pullback square.
\end{theorem}

\begin{proof}
  Consider the factorization
  \[
    \xymatrix{
      B \ar[r]^-j & C \comma v \ar[r]^-{p_1} & C
    }
  \]
  of $v$ introduced in paragraph~\ref{paragr:fact}. The pullback square of
  the statement factors as a composite of two pullback squares
  \[
    \xymatrix{
      u \comma v \pullbackcorner \ar[d]_{j'} \ar[r] & B \ar[d]^j \\
      \cdot \pullbackcorner \ar[d] \ar[r] & C \comma v \ar[d]^{p_1} \\
      u \comma C \ar[r]_{p_2} & C
    }
  \]
  and it suffices to show that these two squares are
  homotopy pullback squares. By Proposition~\ref{prop:j_retr}, the
  \oo-functor $j$ is a strong right oplax transformation retract.
  On the other hand, by Proposition~\ref{prop:comma_fib_prod}, the
  \oo-functor $j'$ can be identified with the \oo-functor
   \[ u \comma (\id{v}, j) : u \comma v \to u \comma p_1, \]
  which, by Proposition~\ref{prop:comma_retr}, is a strong right oplax
  transformation retract as well. It follows from
  Proposition~\ref{prop:retr_Thomason} that both $j$ and $j'$ are Thomason
  equivalences, showing that the top square is a homotopy pullback square.
  As for the bottom square, this follows from the previous proposition.
\end{proof}

\begin{corollary}[Theorem B]
  If $u : A \to B$ is a colocally homotopically constant \oo-functor and
  $b$ is an object of $B$, then the pullback square
  \[
    \xymatrix{
      \cotr{A}{b} \pullbackcorner \ar[d]_{\cotr{u}{b}} \ar[r] & A \ar[d]^u \\
      \cotr{B}{b} \ar[r] & B \pbox{,}
    }
  \]
  where the horizontal arrows are the forgetful \oo-functors,
  is a homotopy pullback square.
\end{corollary}

\begin{proof}
  This is a particular case of the previous theorem.
\end{proof}

\begin{corollary}\label{coro:comma_homot}
  Let
  \[
    \xymatrix{
      A \ar[r]^u & C & B \ar[l]_v
    }
  \]
  be two \oo-functors. If $v$ is colocally homotopically constant,
  then the comma construction $u \comma v$ is canonically the homotopy
  pullback $A \times_C^h B$.
\end{corollary}

\begin{proof}
  By paragraph~\ref{paragr:fact_dual}, the \oo-functor $u$ factors as
  \[
    \xymatrix{
      A \ar[r]^-{j'} & u \comma C \ar[r]^-{p_2} & C \pbox{,}
    }
  \]
  where $j'$ is a strong left oplax transformation retract and hence a
  Thomason equivalence by Proposition~\ref{prop:retr_Thomason}. We thus get
  a commutative diagram
  \[
    \xymatrix{
      A \ar[r]^u \ar[d]_{j'} & C \ar@{=}[d] & B \ar[l]_v \ar@{=}[d] \\
      u \comma C \ar[r]_{p_2} & C & B \ar[l]^v \pbox{,}
    }
  \]
  where the vertical arrows are Thomason equivalences, and the result follows
  from the previous theorem.
\end{proof}

\begin{remark}
  More precisely, one can show that, under the same hypothesis as in the
  previous corollary, the ``$2$-square''
  \[
    \shorthandoff{;:}
    \xymatrix{
      u \comma v
      \ar[d]_{p_1}
      \ar[r]^{p_2}
      &
      B
      \ar[d]^v
      \\
      A
      \ar[r]_u
      &
      C
      \ar@{}[u];[l]_(.30){}="x"
      \ar@{}[u];[l]_(.70){}="y"
      \ar@{<=}"x";"y"_{\kappa}
    }
  \]
  introduced in paragraph~\ref{paragr:def_comma} is a ``homotopy pullback
  $2$-square'' in some appropriate sense (for instance, its topological
  realization is a homotopy pullback in the sense of Mather \cite{Mather}).
\end{remark}

\begin{corollary}\label{coro:thmB}
  If $u : A \to B$ is a colocally homotopically constant \oo-functor and
  $b$ is an object of $B$, then the \oo-category $\cotr{A}{b}$ is
  canonically the homotopy fiber of~$u$ at $b$.
\end{corollary}

\begin{proof}
  This is a particular case of the previous corollary.
\end{proof}

\begin{remark}
  As the comma construction of two $n$-functors is an $n$-category, the four
  previous statements all restrict to $n$-categories. In particular, we
  recover the original Quillen Theorem B and its generalization to
  $2$-categories proven by \hbox{Cegarra}~\cite{CegThmB}. To get direct
  proofs of these results for $n$-categories, all one has to do is to
  change in our proofs all the ``$\infty$'' to ``$n$'' and to
  replace the $m$-th oriental $\On{m}$ appearing in the proof of
  Proposition~\ref{prop:lemma} by its
  $n$-th truncation $\Ontr{m}{n}$, obtained from $\On{m}$
  by keeping only $i$-cells for $i \le n$ and modding out by $(n+1)$-cells.
  Of course, some parts of these proofs get simpler for small~$n$. Most
  notably, for $n = 1$, the map $f_i$ of the proof of
  Proposition~\ref{prop:lemma} can be identified with the nerve of the
  functor $i_\ast$, so that all one has to prove is that $0_\ast : x(0)
  \comma u \to x \comma u$ is a Thomason equivalence, which can be done by
  describing an explicit structure of transformation retract on this functor
  (note that an oplax transformation between $1$-functors is nothing but a
  natural transformation). More generally, for $n = 1$ and $n = 2$, all the
  intermediate constructions and oplax transformations involved in these
  proofs can be defined by using explicit formulas.
\end{remark}

\begin{remark}
  The four previous results were proven for ``under-\oo-categories''.
  They remain valid for ``over-\oo-categories'' defined as
 $\tr{A}{b} = u \comma b$, for $u : A \to B$ an \oo-functor and $b$ an
 object of $B$. This will follow from the equality $\tr{A}{b} =
 (\cotr{A^\o}{b})^\o$ (see paragraph~\ref{paragr:notation} for the notation
 $C^\o$) and the fact, that we will prove with \hbox{Maltsiniotis}
 in~\cite{AraMaltsiNerfs}, that the duality $C \mapsto C^\o$ sends Thomason
 equivalences to Thomason equivalences.

 If one tries to adapt our proofs to ``over-\oo-categories'', one has to
 replace the \oo-functor $0 : \On{0} \to \On{m}$ appearing in the proof of
 Proposition~\ref{prop:lemma} by the \oo-functor $m : \On{0} \to \On{m}$.
 This \oo-functor is both a right oplax transformation
 retract and a right lax transformation retract, but only the lax structure
 is compatible with the structure of right deformation retract of the
 simplicial map $m : \Deltan{0} \to \Deltan{m}$. Therefore, one has to
 replace the use of our ``oplax'' comma construction $u \comma v
 = A \times_C \HomLax(\Dn{1}, C) \times_C B$, for $u :
 A \to C$ and $v : B \to C$ two \oo-functors, by its ``lax'' variant~$u
 \comma' v = A \times_C \HomOpLax(\Dn{1}, C) \times_C B$, which has
 sesquifunctoriality properties with respect to lax transformations instead
 of oplax transformations. This leads to a proof of our results for
 ``over-\oo-categories'' defined as
 % TOCHECK
 \smash{$\trm{A}{b} = u \comma' b$} (see \cite[Remark~6.37]{AraMaltsiJoint}
 for an explanation of this notation). As \smash{$\trm{A}{b} =
 (\cotr{A^\op}{b})^\op$}, where $C \mapsto C^\op$ denotes the duality of
 $\ooCat$ consisting in reversing cells in odd dimension, the results for
 these ``over-\oo-categories'' also follow formally from our results and the
 fact that the duality $C \mapsto C^\op$ sends Thomason equivalences to
 Thomason equivalences, which is a consequence of the existence of a natural
 isomorphism between $N(C^\op)$ and $N(C)^\op$ (see \cite[Proposition
 5.2]{AraMaltsiThmAII}), where $X \mapsto X^\op$ denotes the usual duality
 of simplicial sets.

 Finally, the results for ``under-\oo-categories'' defined as $\cotrm{A}{b} =
 b \comma' u$ will also follow from the fact that $C \mapsto C^\o$ sends
 Thomason equivalences to Thomason equivalences, as $\cotrm{A}{b} =
 ({\trm{A^\o}{b}})^\o$.
\end{remark}

\section{A few applications}

A first consequence of the \oo-categorical Theorem B is the (non-relative)
\oo-categorical Theorem A, which is a special case of the main result of
\cite{AraMaltsiThmAI} and \cite{AraMaltsiThmAII}.

\begin{paragraph}
  We will say that an \oo-category $A$ is \ndef{aspherical} if the unique
  \oo-functor from~$A$ to the terminal \oo-category is a Thomason
  equivalence or, in other words, if its nerve $NA$ is weakly contractible.
\end{paragraph}

\begin{theorem}\label{thm:thmA}
  Let $u : A \to B$ be an \oo-functor. If for every object $b$ of $B$, the
  \oo-category $\cotr{A}{b}$ is aspherical, then $u$ is a Thomason
  equivalence.
\end{theorem}

\begin{proof}
  The hypothesis implies that $u$ is colocally homotopically constant. We
  can thus apply Theorem B and more precisely Corollary~\ref{coro:thmB}.
  We get that, for every object~$b$ of~$B$, the \oo-category
  $\cotr{A}{b}$ is the homotopy fiber of $u$ at $b$. As by hypothesis
  $\cotr{A}{b}$ is aspherical, this implies that all the homotopy fibers of
  $Nu$ are weakly contractible, showing that $Nu$ is a weak equivalence.
\end{proof}

\goodbreak

We will now use the \oo-categorical Theorem B to produce models of
Eilenberg--Mac Lane spaces. We will need the following lemma:

\begin{lemma}
  Let $A$ be an \oo-category and let $a$ and $a'$ be two objects of $A$.
  There exists a canonical isomorphism
  \[ a \comma a' \simeq \Homi_A(a, a')^\o, \]
  natural in $a$ and $a'$, where $a$ and $a'$ are seen as \oo-functors
  $\Dn{0} \to A$ and $C \mapsto C^\o$ denotes the duality introduced in
  paragraph~\ref{paragr:notation}.
\end{lemma}

\begin{proof}
  See \cite[Proposition B.6.2]{AraMaltsiJoint}.
\end{proof}

\begin{theorem}\label{thm:loop}
  Let $A$ be an \oo-category endowed with an object $a$. Suppose that for
  every $1$\nbd-cell $f : a' \to a''$ of $A$ the induced \oo-functor $\Homi_A(a'',
  a)^\o \to \Homi_A(a', a)^\o$ is a Thomason equivalence. Then
  $\Homi_A(a,a)^\o$ is a model for the loop space of $(A, a)$ in the sense
  that $N(\Homi_A(a,a)^\o)$ has the homotopy type of the loop space of $(NA,
  a)$.
\end{theorem}

\begin{proof}
  By the previous lemma, the hypothesis precisely means that the \oo-functor
  $a : \Dn{0} \to A$ is colocally homotopically constant. By
  Proposition~\ref{coro:comma_homot}, we thus get that $a \comma a$ is the
  homotopy pullback of
  \[
    \xymatrix{
      \Dn{0} \ar[r]^a & A & \Dn{0} \ar[l]_a \pbox{,}
    }
  \]
  that is, that $N(a \comma a)$ is the homotopy pullback of
  \[
    \xymatrix{
      \Deltan{0} \ar[r]^a & NA & \Deltan{0} \ar[l]_a \pbox{,}
    }
  \]
  thereby proving the result.
\end{proof}

\begin{remark}
  As mentioned before, we will prove with Maltsiniotis in
  \cite{AraMaltsiNerfs} that the duality $C \mapsto C^\o$ sends Thomason
  equivalences to Thomason equivalences. Therefore the previous theorem
  remains valid if all the dualities appearing in its statement are removed.
\end{remark}

\begin{paragraph}
  Let $(M, +, e)$ be a commutative monoid. For any $n \ge 1$, we define an
  $n$\nbd-category $\Kcat{M}{n}$ in the following way. Its cells are
  \[
  (\Kcat{M}{n})_k =
  \begin{cases}
    \{\ast\} & \text{if $0 \le k < n$,} \\
    M & \text{if $k = n$;}
  \end{cases}
  \]
  the unit of the unique $(n-1)$-cell is the unit $e$ of the monoid; and if
  $x$ and $y$ are $n$-cells, then for any $0 \le j < n$, we set $x \comp_j
  y = x + y$.
\end{paragraph}

\begin{theorem}\label{thm:Kpn}
  For any abelian group $\pi$ and any $n \ge 1$, the \oo-category
  $\Kcat{\pi}{n}$ is a~$\Ktop{\pi}{n}$ in the sense that $N(\Kcat{\pi}{n})$
  is a~$\Ktop{\pi}{n}$.
\end{theorem}

\begin{proof}
  The result is well known for $n = 1$. If $n \ge 2$, then all the $1$-cells
  of $\Kcat{\pi}{n}$ are identities so that the hypothesis of
  Theorem~\ref{thm:loop} is satisfied. We thus get that the loop space of
  $\Kcat{\pi}{n}$ is $\Homi_{\Kcat{\pi}{n}}(\ast, \ast)^\o$, which is
  isomorphic to $\Kcat{\pi}{n-1}$. The result thus follows by induction
  using the fact that (the nerve of) $\Kcat{\pi}{n}$ is connected.
\end{proof}

\begin{remark}
  In \cite{BergerWreath}, Berger proves that the topological realization of
  the so-called cellular nerve of $\Kcat{\pi}{n}$ is a $\Ktop{\pi}{n}$,
  showing that $\Kcat{\pi}{n}$ is a $\Ktop{\pi}{n}$ in a, \forlang{a priori},
  different sense from the previous theorem (see his Corollary 4.3 and his
  Section~4.10).
  % More precisely, Berger shows (see his Corollary 4.3) that the topological
  % realization of a certain $n$-cellular set obtained from the $\Gamma$-set
  % representing the Eilenberg--Mac Lane spectrum associated to $\pi$ is
  % a~$\Ktop{\pi}{n}$, and the concrete description of this cellular set (see
  % the beginning of his Section~4.10) implies that this $n$-cellular set is
  % the cellular nerve of $\Kcat{\pi}{n}$.
  It will follow from the comparison of Street's nerve and the cellular
  nerve, that we will study with Maltsiniotis in \cite{AraMaltsiNerfs}, that
  these two meanings of ``being a $\Ktop{\pi}{n}$'' coincide.
\end{remark}

\begin{theorem}
  Let $\pi$ be a commutative ordered group whose underlying poset is
  aspherical (as a category). Denote by $\pi^+$ its monoid of positive
  elements. Then, for any $n \ge 1$, the \oo-category $\Kcat{\pi^+}{n}$ is a
  $\Ktop{\pi}{n}$.
\end{theorem}

\begin{proof}
  The inclusion $\pi^+ \subset \pi$ induces an \oo-functor $\Kcat{\pi^+}{n}
  \to \Kcat{\pi}{n}$. By the previous theorem, it suffices to prove that
  this \oo-functor is a Thomason equivalence. We will apply Theorem A
  (Theorem~\ref{thm:thmA}). We have to prove that the
  \oo-category~$\cotr{(\Kcat{\pi^+}{n})}{\ast}$ is aspherical. The concrete
  description of the slice \oo-categories given
  in~\cite[paragraph~4.1]{AraMaltsiThmAI}
  shows that this \oo-category can be described in the following way: it is
  an $n$-category whose underlying $(n-1)$-category is $\Kcat{\pi}{n-1}$
  (where $\Kcat{\pi}{0}$ means $\pi$, as a set) and whose $n$-cells are given
  by the order on $\pi$. In particular, for $n = 1$, we get the poset~$\pi$
  seen as a $1$-category. This poset being aspherical by hypothesis, this
  ends the proof of the case $n = 1$. If $n > 1$, then we have isomorphisms
  \[
  \Homi_{\cotr{(\Kcat{\pi^+}{n})}{\ast}}(\ast, \ast)^\o
  \simeq
  \big(\cotr{(\Kcat{\pi^+}{n-1})}{\ast}\big)^\o
  \simeq
  \cotr{(\Kcat{((\pi^\o)^+)}{n-1})}{\ast},
  \]
  where $\pi^\o$ denotes the group $\pi$ equipped with the opposite order.
  As $\pi^\o$ is aspherical as a poset (since $\pi$ is), we can assume by
  induction that the \oo-category
  $\Homi_{\cotr{(\Kcat{\pi^+}{n})}{\ast}}(\ast, \ast)^\o$ is aspherical. We
  can thus apply Theorem~\ref{thm:loop} and we get that the loop space
  of~$\cotr{(\Kcat{\pi^+}{n})}{\ast}$ is aspherical.  This shows that
  $\cotr{(\Kcat{\pi^+}{n})}{\ast}$ is aspherical, as it is obviously
  connected, thereby ending the proof.
\end{proof}

\begin{example}
  The previous theorem applies to commutative ordered groups whose underlying
  poset is directed. In particular, the \oo-category $\Kcat{\N}{n}$ is a
  $\Ktop{\Z}{n}$.
\end{example}

\goodbreak

We end the section with an application to loop spaces of \oo-groupoids.

\begin{paragraph}
  Recall that a (strict) \ndef{\oo-groupoid} is an \oo-category in which
  every $i$-cell for $i > 0$ is strictly invertible (for the composition in
  codimension 1), and that an \oo-functor $f : G \to H$ between
  \oo-groupoids is an \ndef{equivalence of \oo-groupoids} if
  \begin{itemize}
    \item for every object $y$ of $H$, there exists an object $x$ of $G$ and
    a $1$-cell $u(x) \to y$ in~$H$,
    \item for every $i \ge 0$, every pair of parallel $i$-cells $u, v$ in
    $G$ (two $0$-cells being always considered as parallel)
    and every $(i+1)$-cell $\beta : f(u) \to f(v)$ in $H$, there exists
    an $(i+1)$-cell $\alpha : u \to v$ in $G$ and an $(i+2)$-cell $f(\alpha)
    \to \beta$ in $H$.
  \end{itemize}
\end{paragraph}

\begin{proposition}
  An equivalence of \oo-groupoids is a Thomason equivalence.
\end{proposition}

\begin{proof}
  The equivalences of \oo-groupoids are precisely the weak equivalences
  between \oo-groupoids of the so-called folk model category structure on
  $\ooCat$ \cite{LMW} (see also~\cite{AraMetWGrp}). To prove the result, it
  thus suffices, using Ken Brown's lemma, to show that the trivial
  fibrations of the folk model category structure are Thomason equivalences.
  We will see that the nerve of such a trivial fibration is actually a
  trivial Kan fibration. By adjunction, to prove this, it suffices to show
  that the left adjoint $c : \pref{\cDelta} \to \ooCat$ of the nerve functor
  sends the inclusions $\partial\Deltan{n} \hookto \Deltan{n}$, where $n \ge
  0$ and $\partial\Deltan{n}$ denotes the boundary of $\Deltan{n}$, to a
  cofibration of the folk model category structure. The explicit description
  of $c(K)$, where $K$ is a simplicial complex, given
  % TOCHECK
  in~\cite[Section 9]{AraMaltsiCondE}, shows that
  $c(\partial\Deltan{n})$ is the underlying $(n-1)$-category of the
  $n$-category $c(\Deltan{n}) = \On{n}$. In other words, the \oo-functor
  $c(\partial\Deltan{n}) \to c(\Deltan{n})$ corresponds to the free addition
  of the unique non-trivial $n$-cell of $\On{n}$, and is hence, by
  definition, a folk cofibration, thereby proving the result.
\end{proof}

\begin{remark}
  One can actually show that an \oo-functor between \oo-groupoids is an
  equivalence of \oo-groupoids if and only if it is a Thomason equivalence.
\end{remark}

\begin{theorem}
  Let $G$ be a strict \oo-groupoid endowed with an object $x$. Then
  the loop space of~$(G, x)$ is a product of Eilenberg--Mac Lane spaces
  (including the discrete space~$K(E, 0)$, for $E$ a set, as an
  Eilenberg--Mac Lane space).
\end{theorem}

\begin{proof}
  As every $1$-cell of $G$ is invertible, the hypothesis of
  Theorem~\ref{thm:loop} is satisfied and we get that $\Homi_G(x, x)^\o$,
  which is isomorphic to $\Homi_G(x, x)$ as $G$ is an \oo-groupoid, is the
  loop space of $(G, x)$. As all the connected components of a loop space
  are weakly equivalent, to prove the result, it suffices to show that the
  connected component of any object of $\Homi_G(x, x)$ is a product of
  Eilenberg--Mac Lane spaces.
  Consider the object $\id{x} : x \to x$. Its connected
  component is equivalent to the full sub-\oo-groupoid of~$\Homi_G(x, x)$
  whose only object is $\id{x}$.
  This \oo-groupoid is obtained by ``shifting down'' a sub-\oo-groupoid
  $G'$ of $G$ having only one object (namely $x$) and one $1$-cell (namely
  $\id{x}$). Such an \oo-groupoid $G'$ is known to be equivalent to a
  product of the form $\prod_{n \ge 2} \Kcat{\pi_n}{n}$ (see for
  instance \cite[Theorem~4.17]{AraTypHomStr}) and the connected component of
  $\id{x}$ is thus equivalent to $\prod_{n \ge 1} \Kcat{\pi_{n+1}}{n}$. The
  result thus follows from~Theorem~\ref{thm:Kpn}.
\end{proof}

\appendix

\section{A contraction of the oriental}
\label{app:contr_orient}

The purpose of this appendix is to construct the oplax transformation
retract needed in the proof of Proposition~\ref{prop:lemma}.

\begin{paragraph}
  The appendix relies on Steiner's theory of augmented directed complexes
  as developed in \cite{Steiner}. We will recall the minimal
  amount of information needed to follow our arguments and we refer the
  reader to \cite[Section 2]{AraMaltsiJoint} for a comprehensive
  introduction to this theory in the spirit of our paper.

  We will denote by $\Cad$ the category of augmented directed complexes.
  Recall that an \ndef{augmented directed complex} is an augmented
  complex $K$ (of abelian groups in nonnegative degrees) endowed, for every $p
  \ge 0$, with a submonoid $K_p^\ast$ of $K^{}_p$ of \ndef{positive
  $p$-chains}, and that a \ndef{morphism of augmented directed complexes} is a
  morphism of augmented complexes sending positive $p$-chains to positive
  $p$-chains. Similarly, a \ndef{homotopy} between two such morphisms is a
  homotopy, in the classical sense, sending positive $p$-chains to
  positive $(p+1)$-chains.

  To any augmented directed complex, Steiner associates an \oo-category thus
  defining a functor $\nu : \Cad \to \ooCat$. We will not need the precise
  definition of this functor and we will recall all the properties of $\nu$
  we will use.
\end{paragraph}

\begin{paragraph}
  We will denote by $\cn : \pref{\cDelta} \to \Cad$ the \ndef{normalized complex
  functor}. If $X$ is a simplicial set, the underlying augmented complex
  of $\cn X$ is the classical normalized complex (its $p$-chains are freely
  generated by nondegenerate $p$-simplices of $X$) and~$(\cn X)^\ast_p$,
  for $p \ge 0$, consists of $p$-chains with nonnegative coefficients. In
  particular, if $m \ge 0$ and $p \ge 0$, we have $(\cn\Deltan{m})_p \simeq
  \Z^{(B_{m,p})}$ where
  \[
    B_{m, p} = \{(i_0, \dots, i_p) \mid 0 \le i_0 < \cdots < i_p \le m\}.
  \]
  We will call the graded set $\coprod_p B_{m, p}$ the \ndef{base} of
  $\cn\Deltan{m}$. (It is the unique base in some precise sense that we
  will not need.)
\end{paragraph}

\begin{paragraph}\label{paragr:def_orientals}
  By composing the functors
  \[
    \xymatrix{
      \cDelta \ar[r]^-y &
      \pref{\cDelta} \ar[r]^-\cn &
      \Cad \ar[r]^-\nu &
      \ooCat \pbox{,}
    }
  \]
  where $y$ denotes the Yoneda embedding, we get a cosimplicial object
  \[ \cO : \cDelta \to \ooCat. \]
  This is Steiner's definition of Street's \ndef{cosimplicial object of
  orientals}. For $n \ge 0$, the \oo-category $\On{n}$ is the \ndef{$n$-th
  oriental}. The cosimplicial object $\cO$ induces the so-called
  \ndef{Street nerve}
  \[ N : \ooCat \to \pref{\cDelta}, \]
  sending an \oo-category $C$ to the simplicial set $NC : \Deltan{p} \mapsto
  \Hom_{\ooCat}(\On{p}, C)$. This nerve functor admits as a left adjoint
  the left Kan extension of $\cO : \cDelta \to \ooCat$ along the Yoneda
  embedding.
\end{paragraph}

\goodbreak

\emph{From now on, we fix $m \ge 0$.}

\medskip

We will start by showing that $\cn\Deltan{m}$ retracts by deformation on
$\cn\Deltan{0}$ in some appropriate sense.

\begin{paragraph}\label{paragr:def_homotopy_Cda}
  Consider the morphism $\cn(0) : \cn\Deltan{0} \to \cn\Deltan{m}$
  induced by the simplicial morphism $0 : \Deltan{0} \to \Deltan{m}$
  corresponding to the $0$-simplex $0$ of $\Deltan{m}$, and the morphism
  $\cn(r) : \cn\Deltan{m} \to \cn\Deltan{0}$ induced by the
  unique morphism $r : \Deltan{m} \to \Deltan{0}$. By functoriality, we have
  $\cn(r)\cn(0) = \id{\cn\Deltan{0}}$. We will see that $\cn(0)\cn(r)$ is
  homotopic to $\id{\cn\Deltan{m}}$. For~$p \ge 0$, we define
  \[ h_p : (\cn\Deltan{m})_p \to (\cn\Deltan{m})_{p+1} \]
  by
  \[
    h_p(i_0, \dots, i_p) =
    \begin{cases}
      (0, i_0, \dots, i_p) & \text{if $i_0 > 0$,} \\
       0 & \text{if $i_0 = 0$.}
    \end{cases}
  \]
  Adopting the convention that, for $0 \le j_0 \le \cdots \le j_q \le m$, if
  the sequence of the $j_k$ is not strictly increasing then
  $(j_0, \dots, j_q) = 0$ in $(\cn\Deltan{m})_q$, we can simply write
  \[ h_p(i_0, \dots, i_p) = (0, i_0, \dots, i_p). \]
\end{paragraph}

\begin{proposition}\label{prop:h}
  The morphisms $h_p$ introduced in the previous paragraph define a homotopy
  $h$ from $\cn(0)\cn(r)$ to $\id{\cn\Deltan{m}}$.
\end{proposition}

\begin{proof}
  Let $(i_0, \dots, i_p)$ be an element of the base of $\cn\Deltan{m}$.
  Note first that we have
  \[
    \cn(0)\cn(r)(i_0, \dots, i_p) =
    \begin{cases}
      (0) & \text{if $p = 0$,} \\
      0 & \text{otherwise.}
    \end{cases}
  \]
  To prove that $h$ is a homotopy, we distinguish two cases:
  \begin{itemize}
    \item If $p = 0$, then we have
      \[ dh(i_0) = d(0, i_0) = (i_0) - (0) = (i_0) - \cn(0)\cn(r)(i_0). \]
    \item If $p \ge 1$, then we have
      \[
        \begin{split}
          \MoveEqLeft
          dh(i_0, \dots, i_p) + hd(i_0, \dots, i_p)  \\
          & = d(0, i_0, \dots, i_p) + \sum_{k = 0}^p (-1)^k h(i_0, \dots, \hat i_k, \dots, i_p) \\
          & = (i_0, \dots, i_p) + \sum_{k = 0}^{p} (-1)^{k+1} (0, i_0, \dots, \hat
          i_k, \dots, i_p) \\*
          & \phantom{=1} \qquad + \sum_{k = 0}^p (-1)^k (0, i_0, \dots, \hat i_k, \dots, i_p) \\
          & = (i_0, \dots, i_p) - 0 \\
          & = (i_0, \dots, i_p) - \cn(0)\cn(r)(i_0, \dots, i_p).
          \qedhere
        \end{split}
      \]
  \end{itemize}
\end{proof}

\goodbreak

We now recall how such a homotopy induces an oplax transformation.

\begin{paragraph}
  If $K$ and $L$ are two augmented directed complexes, we define their
  tensor product $K \otimes L$ in the following way: the underlying
  augmented complex is the classical tensor product of the underlying
  augmented complexes, and the positive chains are generated by tensor
  products of positive chains. By \cite[Proposition A.19]{AraMaltsiJoint},
  there exists a canonical \oo-functor
  \[ \nu(K) \otimes \nu(L) \to \nu(K \otimes L), \]
  where the tensor product on the left is the Gray tensor product.

  In particular, if $K$ is an augmented directed complex, we get an augmented
  directed complex $\cn\Deltan{1} \otimes K$. Moreover, the \oo-category
  $\nu(\cn\Deltan{1})$ is canonically isomorphic to~$\Dn{1}$ and we thus get
  an \oo-functor $\Dn{1} \otimes \nu(K) \to \nu(\cn\Deltan{1} \otimes L)$.
  One checks that if $L$ is a second augmented directed complex, then
  morphisms $\cn\Deltan{1} \otimes K \to L$ correspond to homotopies between
  morphisms from $K$ to $L$. This implies that if $h$ is a homotopy between
  morphisms from $K$ to $L$, we get an oplax transformation $\nu(h)$ by
  composing
  \[
    \xymatrix{\Dn{1} \otimes \nu(K) \ar[r] &
      \nu(\cn\Deltan{1} \otimes K) \ar[r]^-{\nu(h)} &
      \nu L \pbox{.}
     }
  \]
  Moreover, if $h$ is a homotopy from $f$ to $g$, then $\nu(h)$ is an oplax
  transformation from~$\nu(f)$ to $\nu(g)$.

  In our case of interest, that is, the case where $K = \cn\Deltan{m}$,
  the canonical morphism $\Dn{1} \otimes \On{m} \to \nu(\cn\Deltan{1}
  \otimes \cn\Deltan{m})$ is an isomorphism (see for instance
  \cite[Proposition 7.5 and Theorem A.15]{AraMaltsiJoint}),
  which we will consider as an equality.
\end{paragraph}

\goodbreak

We finally produce the announced structure of oplax transformation retract.

\begin{paragraph}
  We will denote by
  \[ 0 : \On{0} \to \On{m} \quadand r : \On{m} \to \On{0} \]
  the \oo-functors induced by the simplicial maps $0 : \Deltan{0} \to
  \Deltan{m}$ and $r : \Deltan{m} \to \Deltan{0}$. Recall that $\On{0}$ is
  the terminal \oo-category $\Dn{0}$ and that the \oo-functor $0$ corresponds to
  the object $0$ of $\On{m}$. The \oo-functor $r$ is obviously a retraction
  of $0$.

  By applying the considerations of the previous paragraph to the
  homotopy~$h$ of Proposition~\ref{prop:h}, we obtain an oplax transformation
  $\alpha$ from the composite $\On{m} \xto{r} \On{0} \xto{0} \On{m}$ to the
  identity of $\On{m}$. By definition, this oplax transformation is obtained
  by applying the functor $\nu$ to the morphism $\cn\Deltan{1} \otimes
  \cn\Deltan{m} \to \cn\Deltan{m}$, that we will still denote by $h$,
  given by
  \[
    \begin{split}
      h((0) \otimes (i_0, \dots, i_p)) & =
        \begin{cases}
          (0) & \text{if $p = 0$,} \\
          0 & \text{if $p > 0$,}
        \end{cases} \\
      h((1) \otimes (i_0, \dots, i_p)) & = (i_0, \dots, i_p), \\
      h((0,1) \otimes (i_0, \dots, i_p)) & = (0, i_0, \dots, i_p), \\
    \end{split}
  \]
  where $(i_0, \dots, i_p)$ is in the base of $\cn\Deltan{m}$.
\end{paragraph}

\begin{proposition}\label{prop:contr_orient}
  The \oo-functor $0 : \On{0} \to \On{m}$ is a strong left oplax
  transformation retract. More precisely, the pair $(r, \alpha)$, introduced
  in the previous paragraph, is a strong left oplax transformation retract
  structure on $0 : \On{0} \to \On{m}$.
\end{proposition}

\begin{proof}
  This follows from the previous paragraph. (The condition of strongness
  is automatic as the identity of $0$ is the only $1$-cell from $0$ to $0$
  in $\On{m}$.)
\end{proof}

\goodbreak

We end the appendix with a compatibility result, needed in the proof of
Proposition~\ref{prop:lemma}, between the oplax transformation $\alpha$ and a
classical simplicial homotopy.

\begin{paragraph}\label{paragr:desc_h}
  We will denote by
  \[ k : \Deltan{1} \times \Deltan{m} \to \Deltan{m} \]
  the unique simplicial homotopy from the constant endofunctor of
  $\Deltan{m}$ of value~$0$ to the identity of~$\Deltan{m}$. Recall that
  $k$ sends a $p$-simplex $(\phi, \psi) : \Deltan{p} \to \Deltan{1} \times
  \Deltan{m}$ of $\Deltan{1} \times \Deltan{m}$ to the $p$-simplex
  $(0, \dots, 0, \psi(r), \dots, \psi(p))$ of $\Deltan{m}$, where $r$
  denotes the number of $0$ in the sequence $\phi(0), \dots,
  \phi(p)$.
\end{paragraph}

\begin{paragraph}\label{paragr:desc_e}
  We will denote by $\eta : \Deltan{m} \to N(\On{m})$ the adjunction morphism.
  This morphism sends a $p$-simplex $\psi : \Deltan{p} \to
  \Deltan{m}$ of $\Deltan{m}$ to the $p$-simplex $\cO(\psi) : \On{p} \to
  \On{m}$ of $N(\On{m})$. By definition of $\cO$, we have $\cO(\psi) =
  \nu\cn(\psi)$, with $\cn(\psi) : \cn\Deltan{p} \to \cn\Deltan{m}$
  defined on the base of $\cn\Deltan{p}$ by
  \[ \cn(\psi)(i_0, \dots, i_q) = (\psi(i_0), \dots, \psi(i_q)), \]
  where, following the convention introduced in
  paragraph~\ref{paragr:def_homotopy_Cda}, the right member is zero if
  the sequence $\psi(i_0), \dots, \psi(i_q)$ is not strictly increasing.
\end{paragraph}

\begin{paragraph}\label{paragr:desc_Nalpha}
  Let $u, v : A \to B$ be two \oo-functors and let $\beta : u \tod v$ be an
  oplax transformation. Following~\cite[Appendix A]{AraMaltsiThmAII}, we
  define a simplicial homotopy
  \[ N \beta : \Deltan{1} \times NA \to NB \]
  from $Nu$ to $Nv$ in the following way. Let $(\phi, x) : \Deltan{p} \to
  \Deltan{1} \times NA$ be a $p$-simplex of~$\Deltan{1}
  \times NA$. By definition, the homotopy $N\beta$ sends $(\phi, x)$ to the
  $p$-simplex
  \[
    \xymatrix@C=2.5pc{
      \On{p} \ar[r]^-{\nu(g_\phi)} & \Dn{1} \otimes \On{p} \ar[r]^-{\Dn{1}
      \otimes x} & \Dn{1} \otimes A \ar[r]^-\beta & B
    }
  \]
  of $NB$, where
  \[ g_\phi : \cn\Deltan{p} \to \cn\Deltan{1} \otimes \cn\Deltan{p} \]
  is the morphism defined as follows. Let $(i_0, \dots, i_q)$ be an
  element of the base of~$\cn\Deltan{p}$ and denote by $r$ the number of $0$
  in the sequence $\phi(i_0), \dots, \phi(i_q)$. The morphism $g_\phi$ is defined
  by
  \[
     g_\phi(i_0, \dots, i_q) =
     \begin{cases}
       (1) \otimes (i_0, \dots, i_q) & \text{if $r = 0$,} \\
       (0) \otimes (i_0, \dots, i_q) + (0, 1) \otimes (i_1, \dots, i_q) &
       \text{if $r = 1$,} \\
       (0) \otimes (i_0, \dots, i_q) & \text{if $r \ge 2$,} \\
     \end{cases}
  \]
  where $(i_1, \dots, i_q) = 0$ for $q = 0$.
\end{paragraph}

\begin{proposition}\label{prop:trans_compat}
  The square
  \[
    \xymatrix@C=3pc{
      \Deltan{1} \times \Deltan{m} \ar[r]^-{\Deltan{1} \times \eta} \ar[d]_{k}
      &
      \Deltan{1} \times N(\On{m}) \ar[d]^{N\alpha}
      \\
      \Deltan{m} \ar[r]_-\eta
      &
      N(\On{m})
    }
  \]
  commutes.
\end{proposition}

\begin{proof}
  Let $p \ge 0$ and fix $(\phi, \psi) : \Deltan{p} \to \Deltan{1} \times
  \Deltan{m}$ a $p$-simplex of $\Deltan{1} \times \Deltan{m}$. We want to
  compare the two $p$-simplices $\On{p} \to \On{m}$ of $N(\On{m})$
  associated to $(\phi, \psi)$ by the square of the statement. Each of these
  $p$-simplices will be induced by a morphism $\cn\Deltan{p} \to
  \cn\Deltan{m}$ and we will prove that these two morphisms are equal. We
  thus fix $(i_0, \dots,
  i_q)$ an element of the base of $\cn\Deltan{p}$ and we denote by $r$ the
  number of $0$ in the sequence $\phi(i_0), \dots, \phi(i_q)$.

  By paragraphs~\ref{paragr:desc_h} and~\ref{paragr:desc_e}, the
  morphism $\eta k$ sends the $p$-simplex $(\phi, \psi)$ to the $p$-simplex
  $\nu(f) : \On{p} \to \On{m}$ of
  $N(\On{m})$, where the morphism $f :\cn\Deltan{p} \to \cn\Deltan{m}$
  satisfies
  \[ f(i_0, \dots, i_q) = (0, \dots, 0, \psi(i_r), \dots, \psi(i_q)). \]
  In other words, we have
  \[ f(i_0, \dots, i_q) =
    \begin{cases}
      (\psi(i_0), \dots, \psi(i_q)) & \text{if $r = 0$,} \\
      (0, \psi(i_1), \dots, \psi(i_q)) & \text{if $r = 1$,} \\
      0 & \text{if $r \ge 2$.}
    \end{cases}
  \]

  Similarly, by paragraphs~\ref{paragr:desc_e} and~\ref{paragr:desc_Nalpha},
  the morphism $(N \alpha)(\Deltan{1} \times \eta)$ sends $(\phi, \psi)$ to
  the $p$-simplex
  $\nu(g) : \On{p} \to \On{m}$ of $N(\On{m})$, where the morphism $g
  :\cn\Deltan{p} \to \cn\Deltan{m}$ is the composite
  \[
    \xymatrix@C=3pc{
      \cn\Deltan{p} \ar[r]^-{g_\phi} &
      \cn\Deltan{1} \otimes \cn\Deltan{p} \ar[r]^-{\cn\Deltan{1} \otimes
        \cn(\psi)} & \cn\Deltan{1} \otimes \cn\Deltan{m} \ar[r]^-h &
      \cn\Deltan{m} \pbox{.}
    }
  \]
  To compute this composite, we distinguish three cases:
  \begin{itemize}
    \item If $r = 0$, then we have
      \[
        \begin{split}
          h(\cn\Deltan{1} \otimes \cn(\psi))g_\phi(i_0, \dots, i_q) & =
          h(\cn\Deltan{1} \otimes \cn(\psi))((1) \otimes (i_0, \dots, i_q)) \\
          & =
          h((1) \otimes (\psi(i_0), \dots, \psi(i_q))) \\
          & =
          (\psi(i_0), \dots, \psi(i_q)).
        \end{split}
      \]
    \item If $r = 1$, then we have
      \[
        \begin{split}
          \MoveEqLeft
          h(\cn\Deltan{1} \otimes \cn(\psi))g_\phi(i_0, \dots, i_q) \\
          & =
          h(\cn\Deltan{1} \otimes \cn(\psi))((0) \otimes (i_0, \dots, i_q) +
          (0,1) \otimes (i_1, \dots, i_q)) \\
          & =
          h((0) \otimes (\psi(i_0), \dots, \psi(i_q)) + (0,1) \otimes
          (\psi(i_1), \dots, \psi(i_q)) ) \\
          & =
          (0, \psi(i_1), \dots, \psi(i_q)),
        \end{split}
      \]
      where the last equality has to be checked separately in the cases $q =
      0$ and~$q \neq 0$.
    \item Finally, if $r \ge 2$, then we have
      \[
        \begin{split}
          h(\cn\Deltan{1} \otimes \cn(\psi))g_\phi(i_0, \dots, i_q) & =
          h(\cn\Deltan{1} \otimes \cn(\psi))((0) \otimes (i_0, \dots, i_q)) \\
          & =
          h((0) \otimes (\psi(i_0), \dots, \psi(i_q))) \\
          & = 0.
        \end{split}
      \]
  \end{itemize}
  This shows that $f = g$, thereby ending the proof.
\end{proof}

{
% TOCHECK
\renewcommand{\baselinestretch}{1.8}
\bibliography{biblio}
\bibliographystyle{mysmfplain}
}

\end{document}

%% file: macros.tex
\usepackage{xspace}

\newif\iffr

\newcommand\fren[2]{\iffr #1\else #2\fi}

\newcommand\journal\emph

\renewcommand\and{\fren{et}{and}\xspace}

\theoremstyle{plain}
\newtheorem{theorem}{Theorem}[section]
\newtheorem*{theorem*}{Theorem}
\newtheorem{proposition}[theorem]{Proposition}
\newtheorem{lemma}[theorem]{Lemma}
\newtheorem{corollary}[theorem]{Corollary}
\theoremstyle{remark}
\newtheorem{remark}[theorem]{Remark}

\newtheorem{example}[theorem]{Example}

\theoremstyle{definition}

\let\paragraph\undefined
\newtheorem{paragraph}[theorem]{}

\newtheorem*{theoremB}{Theorem B}
\theoremstyle{plain}

\numberwithin{equation}{theorem}

\makeatletter
\newif\ifsection
\preto\section{\sectiontrue}
\preto\subsection{\sectionfalse}
\xapptocmd\@sect{%
  \ifsection
    \numberwithin{theorem}{section}
  \else
    \numberwithin{theorem}{subsection}
  \fi
  \setcounter{theorem}{0}\relax}
  {}{}
\makeatother

\let\forlang\emph
\let\ndef\emph
\let\nbd\nobreakdash

\def\xpoint{\futurelet\@let@token\@xpoint}
\def\@xpoint{%
  \ifx\@let@token.\else
    .%
  \fi
  \xspace}

\newcommand\resp[1]{{\rm (resp.}~#1{\rm )}}

\newcommand\zbox[1]{\makebox[0pt][l]{#1}}
\newcommand\pbox[1]{\zbox{\quad#1}}

\newcommand\quadtext[1]{\quad\text{#1}\quad}

\newcommand\quadand{\quadtext{and}}

\renewcommand\le\leqslant
\renewcommand\ge\geqslant
\renewcommand\epsilon\varepsilon
\renewcommand\phi\varphi

\let\hookto\hookrightarrow
\let\xto\xrightarrow

\newcommand\tod\Rightarrow
\newcommand\tot\Rrightarrow

\newcommand\var{{-}}

\newcommand{\sauf}{\mathchoice{\raise 1.8pt\hbox{${\scriptstyle\kern
2.5pt\smallsetminus\kern 2.5pt}$}}{\raise 1.8pt\hbox{${\scriptstyle\kern
2.5pt\smallsetminus\kern 2.5pt}$}}{\raise
1.8pt\hbox{${\scriptscriptstyle\kern 1.5pt\smallsetminus\kern
1.5pt}$}}{\raise 1.8pt\hbox{${\scriptscriptstyle\kern
1.5pt\smallsetminus\kern 1.5pt}$}}}

\newcommand\HomOpLax{\Homi_{\mathrm{oplax}}}
\newcommand\HomLax{\Homi_{\mathrm{lax}}}

\newcommand\ooCatOpLax{\infty\hbox{\protect\nbd-}\mathcal{C}\mathsf{at}^{}_\mathrm{oplax}}
\newcommand\ooCatLax{\infty\hbox{\protect\nbd-}\mathcal{C}\mathsf{at}^{}_\mathrm{lax}}
\newcommand\ooCatOpLaxGray{\infty\hbox{\protect\nbd-}\mathbb{C}\mathsf{at}^{}_\mathrm{oplax}}
\newcommand\pdfoo{\texorpdfstring{$\infty$}{\unichar{"221E}}}
\newcommand\Z{\mathbb{Z}}
\newcommand\N{\mathbb{N}}

\let\C\relax
\newcommand\C{\mathcal{C}}

\newcommand{\Hom}{\operatorname{\mathsf{Hom}}}
\newcommand{\Homi}{\operatorname{\kern.5truept\underline{\kern-.5truept\mathsf{Hom}\kern-.5truept}\kern1truept}}

\newcommand{\pref}[1]{{\widehat{ #1 }}}
\newcommand\id[1]{1_{#1}}
\renewcommand\o\circ
\newcommand\op{\mathrm{op}}

\newcommand{\Cat}{{\mathcal{C}\mspace{-2.mu}\it{at}}}
\newcommand{\nCat}[1]{{#1}\hbox{\protect\nbd-}\kern1pt\Cat}

\newcommand{\ooCat}{\nCat{\infty}}
\newcommand{\coCat}{{\it{co}\mathcal{C}\mspace{-2.mu}\it{at}}}
\newcommand{\ncoCat}[1]{{#1}\hbox{\protect\nbd-}\kern1pt\coCat}

\newcommand\oo[1]{$\infty$\=/}

\let\comp\ast

\newcommand\Dn[1]{\mathrm{D}_{#1}}

\newcommand\cDelta{\mathbf{\Delta}}

\newcommand\Deltan[1]{\varDelta_{#1}}

\newcommand\cO{\mathcal{O}}

\newcommand\On[1]{\mathcal{O}_{#1}}
\newcommand\Ontr[2]{\mathcal{O}_{#1}^{\le #2}}

\newcommand{\tr}[2]{\mathchoice
  {#1\raise -1.8pt\vbox{\hbox{$\kern -.8pt/#2$}}}
  {#1\raise -1.8pt\vbox{\hbox{$\kern -.8pt/#2$}}\kern .8pt}
  {#1\raise -1.8pt\vbox{\hbox{$\scriptstyle\kern -.8pt /#2$}}}
  {#1\raise -1.8pt\vbox{\hbox{$\scriptscriptstyle\kern -.8pt /#2$}}}}

\newcommand{\trm}[2]{\mathchoice
  {#1\raise -1.8pt\vbox{\hbox{$\kern -.8pt\!\stackrel{\,\rm co}{/}\!\!#2$}}}
  {#1\raise -1.8pt\vbox{\hbox{$\kern -.8pt\!\stackrel{\,\rm co}{/}\!\!#2$}}\kern .8pt}
  {#1\raise -1.8pt\vbox{\hbox{$\scriptstyle\kern -.8pt\!\stackrel{\,\,\rm co}{/}\!\!#2$}}\kern .8pt}
  {TODO}}

\newcommand{\trto}[2]{\mathchoice
  {#1\raise -1.8pt\vbox{\hbox{$\kern -.8pt\!\stackrel{\,\,t \rm o}{/}\!\!\!#2$}}}
  {#1\raise -1.8pt\vbox{\hbox{$\kern -.8pt\!\stackrel{\,\,t \rm o}{/}\!\!\!#2$}}\kern .8pt}
  {#1\raise -1.8pt\vbox{\hbox{$\scriptstyle\kern -.8pt\!\stackrel{\,\,t \rm o}{/}\!\!\!#2$}}\kern .8pt}
  {TODO}}

\newcommand{\cotr}[2]{\mathchoice
  {\raise -1.8pt\vbox{\hbox{$#2\backslash$}}#1}
  {\raise -1.8pt\vbox{\hbox{$#2\backslash$}}#1}
  {\raise -1.8pt\vbox{\hbox{$\scriptstyle#2\backslash$}}#1}
  {\raise -1.8pt\vbox{\hbox{$\scriptscriptstyle#2\backslash$}}#1}}

\newcommand{\cotrm}[2]{\mathchoice
  {\raise -1.8pt\vbox{\hbox{$#2\!\stackrel{\!\rm co}{\backslash}$}}#1}
  {\raise -1.8pt\vbox{\hbox{$#2\!\stackrel{\!\rm co}{\backslash}$}}#1}
  {\raise -1.8pt\vbox{\hbox{$\scriptstyle#2\!\stackrel{\!\rm co}{\backslash}$}}#1}
  {TODO}}

\newcommand\Cyl\Gamma

\newcommand\comma{\operatorname{\downarrow}}

\newcommand{\Cad}{\mathcal{C}_{\mathrm{ad}}}

\newcommand\Zdec{\underline{\mathbb{Z'}\kern -2.5pt}\kern 2pt}
\newcommand\cn{\mathsf{c}} % complexe normalisé

\newcommand\W{\mathcal{W}}

\newcommand{\pullbackcorner}[1][dr]{\save*!/#1-1.2pc/#1:(-1,1)@^{|-}\restore}

\newcommand\Kcat[2]{B^{#2}#1}
\newcommand\Ktop[2]{K(#1,#2)}

\newcommand\timesh{\times^{\textrm{h}}}